\RequirePackage{amsmath}
\documentclass[letterpaper]{article}

\usepackage[T1]{fontenc}
\usepackage[utf8]{inputenc}
\RequirePackage{graphicx}
\RequirePackage[numbers,sort&compress]{natbib}

\usepackage{amsthm}
\newtheorem{theorem}{Theorem}
\newtheorem{remark}{Remark}
\newtheorem{definition}{Definition}
\newtheorem{lemma}{Lemma}

\numberwithin{theorem}{section}
\numberwithin{definition}{section}
\numberwithin{theorem}{section}
\numberwithin{lemma}{section}
\numberwithin{corollary}{section}

\usepackage[rgb]{xcolor}
\usepackage{amsfonts}

\usepackage{mathptmx}
\usepackage{booktabs}

\usepackage{graphicx}

\usepackage{xparse} %
\usepackage[hyphens,spaces,obeyspaces]{url}
\usepackage{hyperref}
\usepackage{cleveref}
\hypersetup{
	pdftitle={Automated Local Fourier Analysis (aLFA)},
	pdfauthor={K. Kahl and N. Kintscher},
	hypertexnames=false
}

\usepackage{amsopn}

\usepackage[noend,boxruled,linesnumbered,algo2e,algosection]{algorithm2e} %

\crefname{algocf}{algorithm}{algorithms}
\Crefname{algocf}{Algorithm}{Algorithms}

\DeclareDocumentCommand{\textproc}{m}{\texttt{#1}}

\SetCommentSty{commentfont}

\newcommand{\SetAlgorithmStyle}{
	\setcounter{AlgoLine}{0}
	\SetKwData{Left}{left}\SetKwData{This}{this}\SetKwData{Up}{up}
	\SetKwProg{Fn}{Function}{}{}
	\SetKwInput{Input}{Input}
	\SetKwInput{Output}{Output}
	\SetKwComment{Comment}{$\triangleright$\phantom{$\triangleright$}}{}
	\SetCommentSty{commentfont}
	\SetKwFor{For}{for}{}{end}
	\SetArgSty{}
	\DontPrintSemicolon
}

\usepackage{bold-extra}
\DeclareDocumentCommand{\func}{m m}{\mathtt{#1}(#2)}

\let\oldnl\nl%
\newcommand{\nonl}{
	\renewcommand{\nl}{\let\nl\oldnl}}%

\makeatletter
\renewcommand{\algocf@caption@boxruled}{%
	\hrule
	\hbox to \hsize{%
		\vrule\hskip-0.4pt
		\vbox{   
			\vskip\interspacetitleboxruled%
			\unhbox\algocf@capbox\hfill
			\vskip\interspacetitleboxruled
		}%
		\hskip-0.4pt\vrule%
	}\nointerlineskip%
}%
\makeatother%

\usepackage{afterpage}

\RequirePackage{dutchcal}

\usepackage{nicefrac}

\makeatletter
\g@addto@macro\@floatboxreset\centering
\makeatother

\usepackage{pgfplots}%
\pgfplotsset{compat=newest}

\usetikzlibrary{shapes.geometric,calc,intersections,plotmarks,patterns,through,angles,quotes,positioning,shapes.misc,cd,decorations.markings,arrows,tikzmark}

\usepackage{mathtools}

\usepackage{placeins}

\usepackage{atbegshi}%

\newcommand{\la}{\mathcal{a}}
\newcommand{\lA}{\mathcal{A}}
\newcommand{\lb}{\mathcal{b}}
\newcommand{\lB}{\mathcal{B}}
\newcommand{\lc}{\mathcal{c}}
\newcommand{\lC}{\mathcal{C}}

\newcommand{\lz}{\mathcal{z}}
\newcommand{\lZ}{\mathcal{Z}}

\newcommand{\spec}{\operatorname{spec}}
\newcommand{\bigslant}[2]{{\raisebox{.2em}{$#1$}\left/\raisebox{-.2em}{$#2$}\right.}}

\newcommand{\set}[2]{\{#1 \, : \, #2\}}

\newcommand{\ip}[3][2]{\langle #2,#3\rangle_{#1}}

\newcommand*\cexp[2]{e^{2\pi i\ip{#1}{#2}}}

\newcommand{\fs}{
	\mathbcal{L}}

\newcommand{\mc}[1]{
	\mathcal{#1}}

\newcommand{\PT}{\mc{P}}

\newcommand{\lop}{L}
\newcommand{\trop}{\mathbb{T}}
\newcommand{\pc}{\Xi}

\newcommand{\se}{\mathfrak{s}} %
\newcommand{\st}{\mathfrak{t}}
\newcommand{\sed}{\mathfrak{d}}%
\renewcommand{\sec}{\mathfrak{c}}%
\newcommand{\see}{\mathfrak{e}}
\newcommand{\sef}{\mathfrak{f}}
\newcommand{\seu}{\mathfrak{u}}
\newcommand{\sev}{\mathfrak{v}}
\newcommand{\sel}{\se}%
\renewcommand{\seu}{\mathfrak{u}} %
\newcommand{\m}[1]{
	\mathbb{#1}}
\newcommand{\mL}{
	\mathbb{L}}

\newcommand{\mR}{
	\mathbb{R}}
\newcommand{\mC}{
	\mathbb{C}}
\newcommand{\mZ}{
	\mathbb{Z}}
\newcommand{\mQ}{
	\mathbb{Q}}
\newcommand{\mN}{
	\mathbb{N}}

\newdimen\XCoord
\newdimen\YCoord

\makeatletter
\renewcommand*
\env@matrix[1][\arraystretch]{%
	\edef
	\arraystretch{#1}%
	\hskip -\arraycolsep
	\let\@ifnextchar\new@ifnextchar
	\array{*\c@MaxMatrixCols c}}
\makeatother

\makeatletter
\newcommand{\providepgfdeclarepatternformonly}[5]{%
	\pgfutil@ifundefined{pgf@pattern@name@#1}{%
		\pgfdeclarepatternformonly{#1}{#2}{#3}{#4}{#5}%
	}{%
	}%
}
\makeatother

\definecolor{hellgrau}{rgb}{.75,.75,.75}

\definecolor{colmath}{rgb}{0.443,0.659,0.816} %
\definecolor{colgs}{rgb}{0.243,0.529,0.757} %

\colorlet{colmathA}{colmath!150>wheel,1,4}
\colorlet{colmathB}{colmath!150>wheel,2,4}
\colorlet{colmathC}{colmath!150>wheel,3,4}
\colorlet{colmathD}{colmath!150>wheel,2,12}
\colorlet{colmathE}{colmath!150>wheel,10,12}
\colorlet{colmathF}{colmath!150>wheel,1,12}
\colorlet{colmathG}{colmath!150>wheel,11,12}
\colorlet{colmathN}{colmath!150>wheel,5,12}

\colorlet{colorFN}{colmath!100}
\colorlet{colorSN}{colmathE!100}
\colorlet{colorTN}{colmathG!50}

\colorlet{MyColMathRed}{colmath!150>wheel,8,20}
\colorlet{DirectionColoring}{MyColMathRed}

\pgfplotscreateplotcyclelist{mycolorlist}{%
	color=red, mark=x\\%
	color=green, mark=triangle\\%
	color=blue, mark=+\\%
}

\colorlet{fxnote}{colmath!150}
\colorlet{fxwarning}{colmathD}
\colorlet{fxerror}{colmathE}
\colorlet{fxfatal}{colmathF}

\DeclareDocumentCommand{\CGcaptionAaa}{ O{} }{
	\raisebox{-1pt}{
		\resizebox{!}{.85em}{
			\begin{tikzpicture}
				\node[LaaA] at (0,0) {$#1$};
			\end{tikzpicture}
		}}}
\DeclareDocumentCommand{
	\CGcaptionBaa}{O{}}{
	\raisebox{-1pt}{
		\resizebox{!}{.85em}{
			\begin{tikzpicture}
				\node[LaaB] at (0,0) {$#1$};
			\end{tikzpicture}
		}}}
\DeclareDocumentCommand{
	\CGcaptionAab}{O{}}{
	\raisebox{-1pt}{
		\resizebox{!}{.85em}{
			\begin{tikzpicture}
				\node[LabA] at (0,0) {\rotatebox{150}{$#1$}};
			\end{tikzpicture}
		}}}
\DeclareDocumentCommand{
	\CGcaptionBab}{O{}}{
	\raisebox{-1pt}{
		\resizebox{!}{.85em}{
			\begin{tikzpicture}
				\node[LabB] at (0,0) {$#1$};
			\end{tikzpicture}
		}}}
\DeclareDocumentCommand{
	\CGcaptionAba}{O{}}{
	\raisebox{-1pt}{
		\resizebox{!}{.85em}{
			\begin{tikzpicture}
				\node[LbaA] at (0,0) {$#1$};
			\end{tikzpicture}
		}}}
\DeclareDocumentCommand{
	\CGcaptionBba}{O{}}{
	\raisebox{-1pt}{
		\resizebox{!}{.85em}{
			\begin{tikzpicture}
				\node[LbaB] at (0,0) {$#1$};
			\end{tikzpicture}
		}}}
\DeclareDocumentCommand{
	\CGcaptionAbb}{O{}}{
	\raisebox{-1pt}{
		\resizebox{!}{.85em}{
			\begin{tikzpicture}
				\node[LbbA] at (0,0) {\rotatebox{-30}{$#1$}};
			\end{tikzpicture}
		}}}
\DeclareDocumentCommand{
	\CGcaptionBbb}{O{}}{
	\raisebox{-1pt}{
		\resizebox{!}{.85em}{
			\begin{tikzpicture}
				\node[LbbB] at (0,0) {$#1$};
			\end{tikzpicture}
		}}}

\newcommand\Item[1][]{%
	\ifx
	\relax#1\relax  \item
	\else
	\item[#1]
	\fi
	\abovedisplayskip=0pt\abovedisplayshortskip=0pt~\vspace*{-\baselineskip}}

\newlength\MyPageWidth
\DeclareDocumentCommand
\TikzFig{ O{1} m }{
	\ifthenelse{\equal{#1}{0}}{
		\noindent\fbox{
			\includegraphics{#2.pdf}}
		\newline
		\settoheight\imageheight{\includegraphics{#2.pdf}}
		h : %
		\printlength{\imageheight}
		\settowidth\imagewidth{\includegraphics{#2.pdf}}
		w : \printlength{\imagewidth}
	}{
		\noindent\includegraphics{#2.pdf}
	}
}

\DeclareDocumentCommand
\TikzFigr{ O{1} O{.5} m }{
	\ifthenelse{\equal{#1}{0}}{
		\noindent\fbox{
			\resizebox{#2\MyPageWidth}{!}{\includegraphics{#3.pdf}}}
		\newline
		\settoheight\imageheight{\includegraphics{#3.pdf}}
		h : %
		\printlength{\imageheight}
		\settowidth\imagewidth{\includegraphics{#3.pdf}}
		w : \printlength{\imagewidth}
	}{
		\noindent{
			\resizebox{#2\MyPageWidth}{!}{\includegraphics{#3.pdf}}}
	}
}

\DeclareDocumentCommand{\TikzScope}{m m}{
	\begin{scope}
		\node[fill=white,fill opacity=.25,text opacity=1,label=above:{\makebox[0pt][c]{#1}
		}] (p1) at (0,0) { {\TikzFig{#2}} };
	\end{scope}
}
\DeclareDocumentCommand{\TikzScopephantom}{m m}{
	\begin{scope}
		\node[fill=white,fill opacity=0,opacity=0,text opacity=0,label=above:{\makebox[0pt][c]{#1}
		}] (p1) at (0,0) {\phantom{\TikzFig{#2}} };
	\end{scope}
}
\DeclareDocumentCommand{\TikzScopephantomresize}{O{.5} m m}{
	\begin{scope}
		\node[fill=white,fill opacity=0,opacity=0,text opacity=0,label=above:{\makebox[0pt][c]{#2}
		}] (p1) at (0,0) {\phantom{\TikzFigr[1][#1]{#3}} };
	\end{scope}
}
\DeclareDocumentCommand{\TikzScopeResize}{ O{.5} m m}{
	\begin{scope}
		\node[fill=white,fill opacity=.25,text opacity=1,label=above:{\makebox[0pt][c]{#2}}] (p1) at (0,0) { \TikzFigr[1][#1]{#3} };
	\end{scope}
}

\DeclareDocumentCommand{\TikzScopeABh}{ m m m m }{
	\begin{scope}
		\TikzScope{#1}{#3}
	\end{scope}
	\begin{scope}[yshift=-.5*\textwidth]
		\TikzScope{#2}{#4}
	\end{scope}
}
\DeclareDocumentCommand{\TikzFigsABh}{ O{a)} O{b)} m m }{
	\begin{tikzpicture}
		\TikzScopeABh{#1}{#2}{#3}{#4}
	\end{tikzpicture}
}
\DeclareDocumentCommand{\TikzScopeAB}{ m m m m }{
	\begin{scope}
		\TikzScope{#1}{#3}
	\end{scope}
	\begin{scope}[xshift=.5*\textwidth]
		\TikzScope{#2}{#4}
	\end{scope}
}

\DeclareDocumentCommand{\TikzScopeABphantom}{ m m m m m }{
	\begin{scope}
		\TikzScope{}{#3}
		\TikzScopephantom{#1}{#5}
	\end{scope}
	\begin{scope}[xshift=.5*\textwidth]
		\TikzScope{}{#4}
		\TikzScopephantom{#2}{#5}
	\end{scope}
}
\DeclareDocumentCommand{\TikzScopeABphantomresize}{ m m m m m }{
	\begin{scope}
		\TikzScopeResize{}{#3}
		\TikzScopephantomresize{#1}{#5}
	\end{scope}
	\begin{scope}[xshift=.5*\textwidth]
		\TikzScopeResize{}{#4}
		\TikzScopephantomresize{#2}{#5}
	\end{scope}
}

\DeclareDocumentCommand{\TikzScopeABr}{ m m m m }{
	\begin{scope}
		\TikzScopeResize{#1}{#3}
	\end{scope}
	\begin{scope}[xshift=.5*\textwidth]
		\TikzScopeResize{#2}{#4}
	\end{scope}
}

\DeclareDocumentCommand{\TikzFigsABphantom}{ O{a)} O{b)} m m m }{
	\begin{tikzpicture}
		\TikzScopeABphantom{#1}{#2}{#3}{#4}{#5}
	\end{tikzpicture}
}
\DeclareDocumentCommand{\TikzFigsABphantomresize}{ O{a)} O{b)} m m m }{
	\begin{tikzpicture}
		\TikzScopeABphantomresize{#1}{#2}{#3}{#4}{#5}
	\end{tikzpicture}
}

\DeclareDocumentCommand{\TikzFigsAB}{ O{a)} O{b)} m m }{
	\begin{tikzpicture}
		\TikzScopeAB{#1}{#2}{#3}{#4}
	\end{tikzpicture}
}

\DeclareDocumentCommand{\TikzFigsABresize}{ O{a)} O{b)} m m }{
	\begin{tikzpicture}
		\TikzScopeABr{#1}{#2}{#3}{#4}
	\end{tikzpicture}
}

\DeclareDocumentCommand{\TikzFigsABCD}{ O{a)} O{b)} O{c)} O{d)} m m m m}{
	\begin{tikzpicture}
		\TikzScopeAB{#1}{#2}{#5}{#6}
		\begin{scope}[yshift=-.45*\textwidth]
			\TikzScopeAB{#3}{#4}{#7}{#8}
		\end{scope}
	\end{tikzpicture}
}

\DeclareDocumentCommand{\TikzFigsABCDtwo}{ O{a)} O{b)} O{c)} O{d)} m m m m}{
	\begin{tikzpicture}
		\TikzScopeAB{#1}{#2}{#5}{#6}
		\begin{scope}[yshift=-.4*\textwidth]
			\TikzScopeAB{#3}{#4}{#7}{#8}
		\end{scope}
	\end{tikzpicture}
}

\DeclareDocumentCommand{\TikzFigsABCDresize}{ O{a)} O{b)} O{c)} O{d)} m m m m}{
	\begin{tikzpicture}
		\TikzScopeABr{#1}{#2}{#5}{#6}
		\begin{scope}[yshift=-.42*\textwidth]
			\TikzScopeABr{#3}{#4}{#7}{#8}
		\end{scope}
	\end{tikzpicture}
}
\DeclareDocumentCommand{\TikzFigsABCDresizeCYShift}{ O{a)} O{b)} O{c)} O{d)} m m m m m}{
	\begin{tikzpicture}
		\TikzScopeABr{#1}{#2}{#5}{#6}
		\begin{scope}[yshift=-#9pt]
			\TikzScopeABr{#3}{#4}{#7}{#8}
		\end{scope}
	\end{tikzpicture}
}

\DeclareDocumentCommand{\TikzFigsABCresize}{ O{a)} O{b)} O{c)} m m m}{
	\begin{tikzpicture}
		\TikzScopeABr{#1}{#2}{#4}{#5}
		\begin{scope}[yshift=-.5*\textwidth]
			\TikzScopeResize{#3}{#6}
		\end{scope}
	\end{tikzpicture}
}

\DeclareDocumentCommand{\TikzFigsXYZresize}{ O{a)} O{b)} O{c)} m m m}{
	\begin{tikzpicture}
		\begin{scope}[xshift=0*\textwidth]
			\TikzScopeResize[.275]{#1}{#4}
		\end{scope}
		\begin{scope}[xshift=.325*\textwidth]
			\TikzScopeResize[.275]{#2}{#5}
		\end{scope}
		\begin{scope}[xshift=.65*\textwidth]
			\TikzScopeResize[.275]{#3}{#6}
		\end{scope}
	\end{tikzpicture}
}

\DeclareDocumentCommand{\TikzFigsABC}{ O{a)} O{b)} O{c)} m m m}{
	\begin{tikzpicture}
		\TikzScopeAB{#1}{#2}{#4}{#5}
		\begin{scope}[yshift=-.5*\textwidth]
			\TikzScope{#3}{#6}
		\end{scope}
	\end{tikzpicture}
}

\DeclareDocumentCommand{\mat}{O{c} m}{
	\ensuremath{\left[\begin{matrix*}[#1]
			#2
		\end{matrix*}\right]}
}

\DeclareDocumentCommand{\exbox}{O{Discretized Laplacian with full weighting} m}{
	\begin{center}
		\begin{tikzpicture}
			\node[rectangle, top color=white!10, bottom color=white!10, rounded corners=5pt, inner xsep=5pt, inner ysep=6pt, outer ysep=10pt,draw=gray]{
				\begin{minipage}{0.95\linewidth}
					\textbf{#1}\\[\medskipamount]
					#2
				\end{minipage}
			};%
		\end{tikzpicture}
	\end{center}
}

\makeatletter
\pgfplotstableset{
	discard if not/.style 2 args={
		row predicate/.code={
			\def\pgfplotstable@loc@TMPd{\pgfplotstablegetelem{##1}{#1}\of}
			\expandafter\pgfplotstable@loc@TMPd\pgfplotstablename
			\edef\tempa{\pgfplotsretval}
			\edef\tempb{#2}
			\ifx
			\tempa\tempb
			\else
			\pgfplotstableuserowfalse
			\fi
		}
	}
}
\makeatother

\pgfplotsset{
	discard if not/.style 2 args={
		x filter/.append code={
			\edef\tempa{\thisrow{#1}}
			\edef\tempb{#2}
			\ifx
			\tempa\tempb
			\else
			\def\pgfmathresult{NaN}
			\fi
		},
	},
}

\pgfplotsset{
	discard if out of range/.style n args={3}{
		x filter/.code={
			\edef\tempa{\thisrow{#1}}
			\edef\tempb{#2}
			\edef\tempc{#3}
			\ifdim
			\tempa pt> \tempb pt
			\ifdim
			\tempa pt< \tempc pt
			\else
			\def\pgfmathresult{inf}
			\fi
			\else
			\def\pgfmathresult{inf}
			\fi
		}
	}
}

\pgfdeclarelayer{background}
\pgfdeclarelayer{backbackbackground}
\pgfdeclarelayer{backbackground}
\pgfdeclarelayer{foreground}
\pgfdeclarelayer{foreforeground}
\pgfsetlayers{backbackbackground,backbackground,background,main,foreground,foreforeground}
\newcommand{\sorttable}[1]{
	\pgfplotstablegetrowsof{#1}
	\pgfplotstablesort[sort key={2},sort cmp={float >}]{\sortedtable}{#1}%
	\pgfplotstablegetelem{1}{2}\of{\sortedtable}%
	\let\Zmax=\pgfplotsretval%
}

\newcommand{\findZmax}[1]{
	\pgfplotstableread{#1}{\mytable}
	\pgfmathsetmacro{\mymax}{0}
	\pgfplotstablegetrowsof{\mytable}
	\pgfmathtruncatemacro{\NumRows}{\pgfplotsretval-1}
	\pgfplotsforeachungrouped \i in {0,...,\NumRows}{
		\pgfplotstablegetelem{\i}{2}\of{\mytable}
		\pgfmathsetmacro{\myvalue}{\pgfplotsretval}
		\pgfmathparse{(\myvalue>\mymax)?1:0}
		\ifdim
		\pgfmathresult pt>0pt \pgfmathsetmacro{\mymax}{\pgfplotsretval}
		\fi
	}
	\let\Zmax=\mymax%
	\pgfmathsetmacro{\zmd}{\mymax - .0002}
	\let\ZmaxDomainMin=\zmd%
}
\pgfmathsetmacro{\atomsize}{.2}
\pgfmathsetmacro{\atomsizesmall}{.1}
\pgfmathsetmacro{\coarsescale}{2.5}
\tikzstyle{atom}=[circle, inner sep=0 cm, minimum width=\atomsize cm]
\tikzstyle{catom}=[circle, inner sep=0 cm, minimum width=1.3*\atomsize cm]
\tikzstyle{smallatom}=[circle, inner sep=0 cm, minimum width=\atomsizesmall cm]
\tikzstyle{csmallatom}=[circle, inner sep=0 cm, minimum width=1.5*\atomsizesmall cm]
\tikzstyle{typeA}=[draw=black, fill=white]
\tikzstyle{typeB}=[draw=black, fill=black!50!white]
\tikzstyle{bond}=[draw=black,thick,shorten >=.5*\atomsize cm,shorten <=.5*\atomsize cm]
\tikzstyle{bondthin}=[draw=black,thick,shorten >=.5*\atomsize cm,shorten <=.5*\atomsize cm,very thin, dotted]
\tikzstyle{bondsmallatom}=[draw=black!50!white,thick,shorten >=.5*\atomsizesmall cm,shorten <=.5*\atomsizesmall cm]
\tikzstyle{coarse}=[draw=colmath!150,very thick,shorten >=\atomsize cm,shorten <=\coarsescale*\atomsize/2 cm]
\tikzstyle{zigzag}=[draw=colmathA,very thick,loosely dashdotted]
\tikzstyle{armchair}=[draw=colmathC,very thick,loosely dotted]

\tikzstyle{he}=[regular polygon,regular polygon sides=6,inner sep=0, minimum width=\coarsescale*\atomsize cm,draw=black,thick]
\tikzstyle{tr}=[regular polygon,regular polygon sides=3,rotate=30,inner sep=0, minimum width=\coarsescale*\atomsize cm,draw=black,thick]
\tikzstyle{st}=[shape=star,star points=6, star point ratio=1.5, inner sep=0,rotate=30, minimum width=\coarsescale*\atomsize cm,draw=black,thick]
\tikzstyle{ci}=[circle, inner sep=0, minimum width=\coarsescale*\atomsize cm,draw=black,thick]

\tikzstyle{BZ}=[pattern color=colmathA, pattern=BZNEL]
\tikzstyle{DG}=[pattern color=colmath!150, pattern=north west lines]

\tikzstyle{Gzero}=[ci,fill=colmath!50]
\tikzstyle{Gone}=[tr,fill=colmathA!50]
\tikzstyle{Gtwo}=[he,fill=colmathB!50]
\tikzstyle{Gthree}=[st,fill=colmathC!50]

\tikzstyle{prev}=[ci,fill=colmath!50]
\tikzstyle{curr}=[st,fill=colmathA!50]
\tikzstyle{next}=[he,fill=colmathC!50]

\tikzstyle{LaaA}=[ci,fill=colmath!50]
\tikzstyle{LaaB}=[ci,fill=colmath!50!white!75!black]
\tikzstyle{LabA}=[tr,fill=colmathA!50,rotate=180]
\tikzstyle{LabB}=[tr,fill=colmathA!50!white!75!black]
\tikzstyle{LbaA}=[he,fill=colmathB!50]
\tikzstyle{LbaB}=[he,fill=colmathB!50!white!75!black]
\tikzstyle{LbbA}=[st,fill=colmathC!50]
\tikzstyle{LbbB}=[st,fill=colmathC!50!white!75!black]

\tikzstyle{wshort}=[colmathA!50!black,bend right=15,-latex,very thick]
\tikzstyle{wlong}=[colmathC!50!black,bend right=30,-latex,very thick,dashed]
\tikzstyle{wlongtilde}=[colmathC!50!black,bend right=30,-latex,very thick,dotted]
\tikzstyle{wlong3}=[colmathC!50,bend right=30,-latex,very thick,dashdotted]

\pgfmathsetmacro{\myout}{180-80}
\pgfmathsetmacro{\myin}{180-20}
\tikzstyle{selfloopTL}=[out=\myout,in=\myin,looseness=13]
\tikzstyle{selfloopTR}=[out=80,in=20,looseness=13]

\tikzstyle{selfloopstyle}=[-latex,very thick]
\tikzstyle{wshort2}=[colmathA!50!black,bend right=15,-latex,very thick]
\tikzstyle{wlong2}=[colmathC!50!black,bend right=30,-latex,very thick,dashed]

\tikzstyle{LFAb}=[draw=black,-latex, thick]
\tikzstyle{LFAbfont}=[font=\tiny]
\tikzstyle{LFAaxis}=[draw=black,thin,-latex]
\tikzstyle{LFAaxisfont}=[font=\tiny]
\tikzstyle{LFAaxistickfont}=[font=\tiny]
\def\thisrowunavailableloadtabledirectly{thisrow_unavailable_load_table_directly}
\def\sudogetthisrow#1#2{
	\edef#2{\thisrow{#1}}
	\ifx
	#2\thisrowunavailableloadtabledirectly %
	\getthisrow{#1}{}
	\let#2=\pgfplotsretval
	\fi
}

\pgfset{
	foreach/parallel foreach/.style args={#1in#2via#3}{
		evaluate=#3 as #1 using {{#2}[#3-1]}
	},
}

\title{Automated Local Fourier Analysis (aLFA)\thanks{This work was partially funded by Deutsche Forschungsgemeinschaft (DFG) Transregional Collaborative Research Centre 55 (SFB/TRR55)}}

\author{Karsten Kahl\footnote{School of Mathematics and Natural Sciences, University of Wuppertal, 42097 Germany, \texttt{\{kkahl,kintscher\}@math.uni-wuppertal.de}} \and Nils Kintscher\footnotemark[2]}

\begin{document}
	
	\maketitle

	\begin{abstract}
		Local Fourier analysis is a commonly used tool to assess the quality and aid in the construction of geometric multigrid methods for translationally invariant operators. In this paper we automate the process of local Fourier analysis and present a framework that can be applied to arbitrary, including non-orthogonal, repetitive structures. To this end we introduce the notion of crystal structures and a suitable definition of corresponding wave functions, which allow for a natural representation of almost all translationally invariant operators that are encountered in applications, e.g., discretizations of systems of PDEs, tight-binding Hamiltonians of crystalline structures, colored domain decomposition approaches and last but not least two- or multigrid hierarchies. Based on this definition we are able to automate the process of local Fourier analysis both with respect to spatial manipulations of operators as well as the Fourier analysis back-end. 
		This automation most notably simplifies the user input by removing the necessity for compatible representations of the involved operators. Each individual operator and its corresponding structure can be provided in any representation chosen by the user.
	\end{abstract}

	\section{Introduction}
	Local Fourier analysis (LFA) is a powerful tool used in the construction and analysis of multigrid methods introduced in \cite{Bran1977}. The fundamental idea of LFA is to leverage the connection between position space and frequency space via the Fourier transform. That is, in case the involved operators can be described by stencils in position space, meaning that they are translationally invariant, their Fourier transform yields so-called symbols, which can be handled much more easily.
	In the context of multigrid methods LFA can be used to obtain precise approximations of the asymptotic convergence rate by assessing the spectral radius of the corresponding error propagation operator~\cite{Saad}. This approximation of the convergence rate is (asymptotically) still valid if the translational invariance is slightly violated, as for example when lexicographic Gauss-Seidel type smoothers are used, or in the case of certain non-periodic boundary conditions \cite{Bran1994, MR3866073, Stevenson1990}. In other cases a similar convergence rate can usually be obtained by additional processing \cite{KahlKint2018,TrotOostSchu2001}.
	Due to these facts, LFA is one of the main tools in the quantitative analysis of two- and multigrid methods.

	An introduction to LFA including several examples can be found in~\cite{TrotOostSchu2001,WienJopp2004}. Multigrid methods have first been considered for the solution of the linear systems of equations originating in the discretization of (elliptic) partial differential equations (PDEs)~\cite{TrotOostSchu2001}. Due to the fact that the simplest tiling of space is rectangular and discretizations are particularly simple to carry out on such tilings, the usual multigrid components and the LFA have originally been designed and tailored for such discretizations (cf.~\cite{TrotOostSchu2001,WienJopp2004}). Several other geometries, including systems of PDEs, have been considered in the past as well. LFA has been carried out for operators defined on triangular tilings in~\cite{GaspGracLisb2009} and on hexagonal tilings in~\cite{ZhouFult}. Further, it has been applied to edge-based quadrilateral discretizations \cite{BoonLentVand2008}, regular Voronoi meshes associated with acute triangular grids \cite{RodrSaliGaspLisb2014}, edge-based discretizations on triangular grids \cite{rodrigo_sanz_gaspar_lisbona_2015} and jumping coefficients on rectangular grids \cite{RitExtending2017,RittichBolten2018}. These papers do a complete two-grid analysis, and in some cases even a three-grid analysis, which was first introduced in~\cite{WienOost2001}. 
	
	While the concept of LFA is well understood, its application quickly becomes complex and involved %
	the more frequencies get intermixed, e.g., by block smoothers, in a three-grid analysis, or in higher dimensional problems ($n > 2$). Thus, there exists software that automates the application of the LFA~\cite{lfalab_HR,WienJopp2004}. In contrast to the software described in~\cite{WienJopp2004}, which basically consists of a collection of templates corresponding to certain smoother and restriction/prolongation strategies for specific problems, the software~\cite{lfalab_HR}, freely available on GitHub,\footnote{\href{https://github.com/hrittich/lfa-lab}{\texttt{github.com/hrittich/lfa-lab}}} can be used to analyze arbitrary translationally invariant operators on rectangular grids. This software has for example been used to analyze colored block Jacobi methods in combination with aggressive coarsening applied to PDEs with jumping coefficients in~\cite{RitExtending2017, RittichBolten2018}. As the number of frequencies which get intermixed increases with the block-size of the smoother and the growing coarsening rate, a manual analysis of this problem would be laborious. %

	In contrast to the approach developed in this paper, the bulk of the analysis in the previously mentioned references is mainly carried out in frequency space. 
	LFA with an emphasize on position space is rare in the literature. 
	In \cite{Huckle2008} under the name \emph{compact Fourier analysis}, a position space oriented LFA is introduced where block Toeplitz matrices are used in order to capture the different classes of unknowns.
	This work has some aspects in common with the approach to LFA we develop in this paper, but lacks generality as it is limited to simple systems on rectangular grids. 
	
	The LFA presented in this paper unifies the position space oriented approaches and allows the treatment of operators on arbitrary repetitive structures. To do so we introduce a mathematical framework for the analysis of translationally invariant operators which alter value distributions on lattices and crystals~\cite{Ashcroft, Schrijver:1986:TLI:17634}. These structures can be based on arbitrary sets of primitive vectors, including non-orthogonal ones, i.e., non-rectangular structures. %
	These crystal structures, which naturally occur in the tight-binding descriptions of solid-state physics \cite{Ashcroft} or discretizations of systems of PDEs, enable the convenient and concise description of the resulting operators and allow for the automatic generation of their representations when enlarging their translational invariance, e.g., coarsening in the context of multigrid methods. Furthermore, they are very helpful in the representation of overlapping and non-overlapping block smoothers.
	Our framework is developed to such an extent that the only task required of the user is to provide a description of the occurring operators with respect to (potentially non-matching) descriptions of the underlying repetitive structures, i.e., each operator can be supplied in the simplest or most convenient representation.
	The remainder of the analysis can then be carried out automatically. In contrast to previously developed LFA, this is achieved by explicitly including a connection of the operator to its underlying structure. This allows us on one hand to automate the transformation of operators in position space, e.g., by finding a least common lattice of translational invariance of two operators and to rewrite their representations accordingly. On the other hand, this focus on structure yields a natural representation and discretization of the dual space that enables the automation of the frequency space part of the analysis as well. All these tasks can be carried out using basic principles and normal forms of integer linear algebra~\cite{Schrijver:1986:TLI:17634}. In combination these developments alleviate the use of LFA by removing any manual calculations. That is, neither a mixing analysis nor transformations of operators to common (and rectangular) translational invariances have to be carried out by hand. While our automated LFA does not necessarily enlarge the set of methods that are analyzable by conventional LFA, it enables the reliable and easy-to-use analysis of complex methods on complicated structures (e.g., overlapping block smoothers and discretizations of systems of PDEs). An open-source Python implementation of the automated LFA framework~\cite{aLFA_NK} is freely available on GitLab.\footnote{\href{https://gitlab.com/NilsKintscher/alfa}{\texttt{gitlab.com/NilsKintscher/alfa}}}
	
	The automation presented in this paper does have some limitation in terms of the smoothers that can be analyzed. Any sequential, i.e., lexicographic, smoother with overlapping update regions changes values in the overlap multiple times in one application. This cannot be easily translated to the structures introduced in this paper. While such smoothers have been analyzed in frequency space before (cf.~\cite{MacLOost2011,Molenaar1991,Sivaloganathan:1991:ULM}), this particular treatment of sequential overlap is momentarily not covered in our framework. %
	Note, that the mere presence of overlap is not the problem here. By introducing a coloring, such that the complete sweep can be split into a sequence of updates where each one of them only changes values at most once, automated LFA can be applied as we are going to demonstrate in this paper. Due to the fact that a coloring in overlapping approaches also favors parallelism over their sequential counterparts, we feel that this limitation is relatively minor when targeting actual applications.
	
	The paper is organized as follows. In \cref{sec:lattices_crystals} we introduce basic notation to describe the underlying structures of the operators we want to analyze: lattices and crystals. In the context of LFA we are interested in translationally invariant operators which alter value distributions on crystals. These kind of operators and their properties are specified in \cref{sec:operators}. After that, \cref{sec:example} illustrates the introduced notation by means of the discretized Laplacian and the red-black Gauss-Seidel method, which is the simplest method where a difference to the conventional LFA becomes apparent. In here, we manually rewrite the discretized Laplacian with respect to the translational invariance of the red-black splitting, such that the method can be analyzed via the results of \cref{sec:operators}. \Cref{sec:crystal_reps_and_isos} contains the general theoretical results which are used to automate the complete procedure, removing the need to manually modify operators. In here, several results of integer linear algebra are used, which we review in \cref{sec:sublat_quotient}. 
	Using these results we obtain the algorithms given in \cref{sec:algorithms} which, in combination with the arithmetic of multiplication operators given in \cref{sec:rules_of_comp}, realize the automated LFA.
	Finally, \cref{sec:Application} contains selected examples to demonstrate the merits of the developed approach. First, we analyze a $4$-color overlapping block Gauss-Seidel smoother for the tight-binding Hamiltonian of the carbon allotrope graphene and give some two-grid convergence results. Second, we reproduce the two-level analysis for the curl-curl problem found in~\cite{BoonLentVand2008} as a further illustration of how our approach is applied to complex error propagators and to double-check its results.
	
	\section{Lattices and Crystals}\label{sec:lattices_crystals}
	In order to be able to automate the process of LFA within a unified framework
	we first have to review some basic definitions of integer linear algebra and crystallography~\cite{Ashcroft,Schrijver:1986:TLI:17634}. An (ideal) crystal is an infinite repetition, defined by a lattice, of a structure element. 
	
	\begin{definition}
		Let $\la_1,\la_2,\ldots,\la_n \in \mR^n$ be linearly independent.
		An $n$-dimensional lattice $\mL$ is the set of points
		\begin{align*}
			\mL=\set{x=\sum_{\ell=1}^n j_\ell \la_\ell \in \mR^n}{j_1, j_2,\ldots, j_n \in \mZ}.
		\end{align*}
		The vectors $\la_1,\la_2,\ldots,\la_n$ are known as the \emph{primitive vectors} or \emph{lattice basis}. Using matrix notation, i.e., $\lA:=\mat[c]{\la_1&\la_2&\ldots&\la_n}$, we can abbreviate the notation by
		$$ \mL(\lA) :=  \lA\mZ^n  = \mL.$$
	\end{definition}

	Without loss of generality we restrict the definition of the second component of a crystal, the structure element, to primitive cells of the lattice. 
	\begin{definition}
		\label{def:primitive_cell}
		A \emph{primitive cell} $\pc=\pc(\lA ) \subset \mR^n$ of a lattice $\mL(\lA )$ is a (connected) volume of space that, if translated by all vectors of $\mL$, fills up $\mR^n$ completely without any overlap, i.e.,  $$\dot\cup_{x\in\mL} \set{x + \xi }{\xi \in \pc} =\mR^n.$$%
		
		A common choice of primitive cells is given by
		$$\PT(\lA ):=\lA  [0,1)^n=\set{y\in\mR^n}{ y = \sum_{\ell=1}^n \alpha_\ell \la _\ell,\, 0 \leq \alpha_\ell < 1},$$
		i.e., parallelotopes spanned by the primitive vectors of $\mL(\lA)$. %
	\end{definition}
	
	The structure element of a crystal is now defined to consist of points $\se_{1},\ldots,\se_{m}$ that are contained in a particular primitive cell $\pc$. %
	
	\begin{definition}
		Let $\mL(\lA )$ be a lattice and $\se \in \pc(\lA )^m$, $m \in \mathbb{N}$ be the \emph{structure element}.
		A \emph{crystal} is defined as the set of tuples
		\begin{align*}
			\mL^{\se}(\lA ):= \set{(x+\se_1,x+\se_2,\ldots,x+\se_m)}{x\in\mL(\lA ),\  \se=(\se_1,\ldots,\se_m)}.
		\end{align*}
		The elements of $\mL^{\se}(\lA )$ are collectively written as $x+\se =(x+\se_1,x+\se_2,\ldots,x+\se_m)$.
	\end{definition}
	
	\begin{remark}
		We define a crystal and the associated structure element to be a tuple instead of a set as we want to study value distributions on crystals and particular operators which manipulate them. For this purpose, the order of a structure element is of importance.
	\end{remark}
	
	To give an idea of the various occurrences of crystal structures in numerical applications we illustrate typical examples in~\cref{fig:crystal_examples} together with their crystal representation. There are three main sources of repetitive structures that are well suited for crystal representations. First, the discretization of systems of PDEs on lattices lead to crystal structures, where the different species of unknowns typically constitute the structure element. Second, tight-binding Hamiltonian formulations in solid state physics for crystalline materials naturally imply a crystal representation based on the atomic structure. Last, colored domain decomposition approaches (e.g., red-black Gauss-Seidel) can be easily represented using crystals, where the smallest structure element typically consists of the union of one domain of each color.
	
	\begin{figure}
		\TikzFigsXYZresize[a) System of PDEs][b) Natural][c) Artificial]{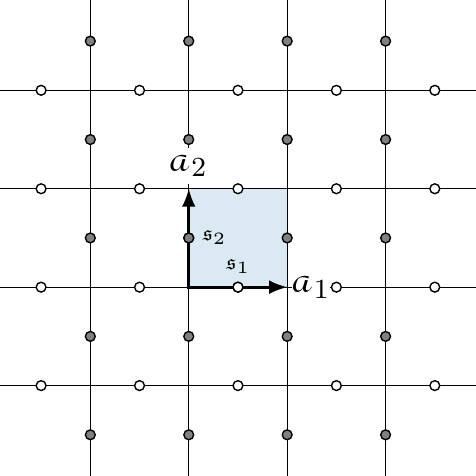}{graphene-triangular-2}{red-black-2}
		\caption[Typical occurrences of crystal representations.]{Illustration of the crystal representation of three repetitive structures of different origin each consisting of primitive vectors $\la_1,\la_2$, a shaded primitive cell $\PT(\lA)$ and a structure element $(\se_1,\se_2)$. a) Staggered discretization of the curl-curl system of PDEs with unknowns on the edges of a rectangular lattice. b) Hexagonal lattice of the carbon allotrope graphene. c) Checkerboard coloring of a rectangular lattice as it is encountered in the red-black Gauss-Seidel smoother.}\label{fig:crystal_examples}
	\end{figure}

	\subsection{Sublattices and quotient spaces}\label{sec:sublat_quotient}
	There are infinitely many representations of a crystal $\mL^\se(\lA)$. On one hand, the representation of any lattice in $n>1$ dimensions is non-unique, i.e., there exist different sets of primitive vectors that yield the same lattice structure. On the other hand, different representations of the crystal can be obtained by shifting the structure element or manipulating the underlying lattice structure, e.g., by using integer linear combinations of the primitive vectors and adjusting the structure element accordingly. 
	
	\begin{figure}
		\centering
		\resizebox{.8\textwidth}{!}{\includegraphics{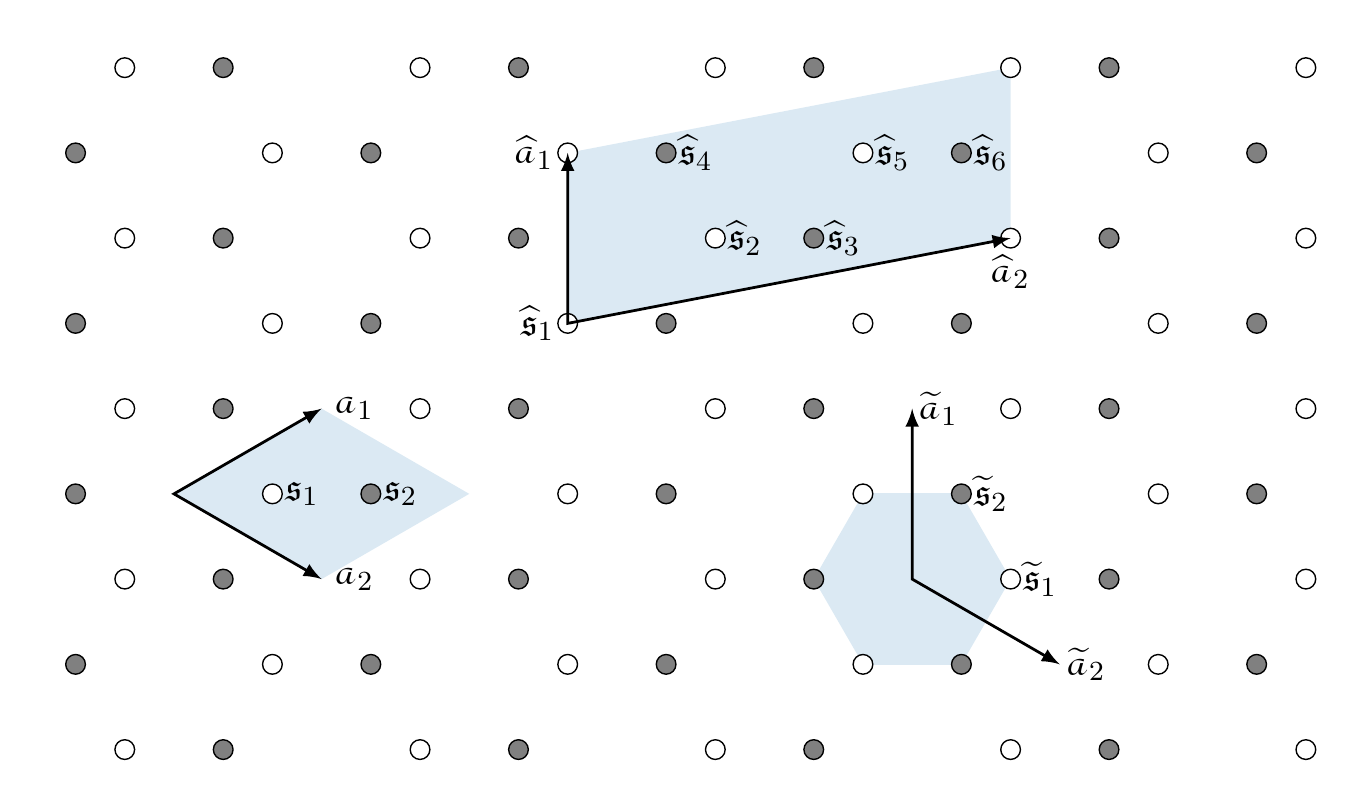}}
		\caption{The structure of graphene: Illustration of three choices of primitive vectors $\lA$, $\widehat{\lA}$ and $\widetilde{\lA}$ with associated structure elements $\se , \widehat{\se}$ and $\widetilde{\se}$ such that  $\mL^{{\se}}(\lA)=\mL^{\widehat{\se}}(\widehat{\lA})=\mL^{\widetilde{\se}}(\widetilde{\lA})$. Furthermore, three different primitive cells are illustrated: The corresponding parallelotopes $\PT(\lA)$,  $\PT(\widehat{\lA})$ as well as a hexagon corresponding to the Voronoi cell of $\mL(\widetilde{\lA})$.}
		\label{fig:crystal_example}
	\end{figure}
	
	In order to cope with this lack of uniqueness of the representation of crystals, illustrated in \cref{fig:crystal_example}, we introduce basic results from integer linear algebra~\cite{Schrijver:1986:TLI:17634}, which resolve the relationship of lattice structures.\footnote{All proofs of~\cref{lem:sublattice,thm:integer_normalforms,thm:lattice_equiv_unimod} can be found in~\cite{Schrijver:1986:TLI:17634}.} An important tool in this is the notion of a sublattice.
	
	\begin{definition}
		Let $\mL(\lA)$ and $\mL(\lC)$ be two lattices, if $\mL(\lA) \supset \mL(\lC)$ then $\mL(\lC)$ is called \emph{sublattice} of $\mL(\lA)$.
	\end{definition}
	
	\begin{lemma}
		\label{lem:sublattice}
		A lattice $\mL(\lC)$ is a sublattice of $\mL(\lA)$ if and only if $\lA^{-1} \lC \in \mZ^{n \times n}$. %
	\end{lemma}
	
	A key-role in the computational comparison of lattices play the \emph{Hermite and Smith normal forms}. The Hermite normal form defines a canonical choice of primitive vectors, whereas the Smith normal form allows us to find least common sublattices. Using these canonical forms is a crucial ingredient in both the theoretical analysis and the automation of the LFA. 
	
	\begin{definition}\label{def:unimodular}
		A matrix $U \in \mZ^{n\times n}$ is called \emph{unimodular} if $\det(U)\in \{\pm 1\}$.
	\end{definition}
	
	\begin{definition}\label{def:HNF}
		A matrix $H\in \mR^{n\times n}$ is in \emph{Hermite normal form} (HNF) if 
		it is upper triangular, elementwise non-negative and its row-wise maximum is located on the diagonal.
	\end{definition}
	
	\begin{definition}\label{def:SNF}
		A matrix $S\in \mR^{n\times n}$ is in \emph{Smith normal form} (SNF) if it is a diagonal matrix and the diagonal entries satisfy $ \nicefrac{s_{i+1}}{s_{i}}\in \mZ$ for all $i=1,\ldots,n-1$. These entries are called the \emph{elementary divisors}.%
	\end{definition}
	
	\begin{theorem}\label{thm:integer_normalforms}
		Let $\lA \in \mQ^{n\times n}$, then
		\begin{enumerate}
			\item[(i)] there exists a unimodular $U \in \mZ^{n \times n}$ such that $H = \lA U$ is in HNF,  %
			\item[(ii)] there exist unimodular $U,V \in \mZ^{n \times n}$ such that $S = V \lA U$ is in SNF. %
		\end{enumerate}
		In addition both normal forms $H$ and $S$ are unique with respect to $\lA$.
	\end{theorem}
	
	\begin{remark}
		Polynomial algorithms to compute the HNF and SNF can for example be found in \cite{ComputerAlgebraHandbook}. Implementations for the computation of these normalforms are for example part of the PARI software package~\cite{PARI2}.
	\end{remark}
	
	Using these definitions and results one obtains a precise statement about the equality of two lattices.

	\begin{theorem}
		\label{thm:lattice_equiv_unimod}
		Let $\mL(\lA)$ and $\mL(\lC)$ be two lattices then the following statements are equivalent.
		\begin{enumerate}
			\item[(i)] $\mL(\lA) = \mL(\lC)$.
			\item[(ii)] $\lA^{-1} \lC  \in \mZ^{n\times n}$ and $\lC^{-1} \lA  \in \mZ^{n\times n}$.
			\item[(iii)] There exists a unimodular matrix $U$, such that $\lC=\lA U$.
			\item[(iv)] $\lA$ and $\lC$ have the same HNF.%
		\end{enumerate}
	\end{theorem}
	
	Instead of analyzing infinite lattice/crystal structures, we limit ourselves to analyze finite dimensional periodic structures due to the fact that we are ultimately aiming to analyze finite dimensional problems. To this end, another helpful tool is the definition of crystal tori which are defined as quotient groups.
	\begin{definition}
		Let $\mL^\se(\lA )$ be a crystal and $\mL(\lC ) \subset \mL(\lA )$ be a sublattice. We define the \emph{crystal torus} $T^\se_{\lA ,\lC }$ by
		\begin{align*}
			T^\se_{\lA ,\lC }:=\bigslant{\mL^\se(\lA )}{\mL(\lC )}.
		\end{align*}
		For every $x+\se \in \mL^\se(\lA )$, their equivalence class $[x+\se]$ is in $T^\se_{\lA ,\lC }$. 
		Furthermore, the elements of $T^\se_{\lA ,\lC }$ are defined by the equivalence
		\begin{align*}[x+\se] = [y+\se] \quad \Longleftrightarrow \quad \text{there exists } z \in \mL(\lC), \text{ such that } x = y + z. %
		\end{align*}
	\end{definition}
	
	For theoretical and practical reasons, e.g., in~\cref{thm:wavefunction_ONB_vec,alg:elements_in_quotient_space},  it is necessary to be able to list all elements of a torus $T^\se_{\lA , \lA  M } = \set{[x]\in T^\se_{\lA , \lA  M } }{x\in \PT(\lA M) \cap \mL^\se(\lA)}$, $M \in \m\mZ^{n \times n}$, uniquely.

	To illustrate this point, consider for example an arbitrary lattice $\mL(\lA)$ with $\lA \in \mathbb{R}^{2\times 2}$, and the lattice torus %
	$T_{\lA , \lA  M }$ with \[
	M = \mat[c]{m_1 & m_2} =
	\mat[c]{
		2 & 3 \\ 2 & -2
	},\ m_{1},m_{2} \in \mathbb{Z}^{2},
	\]
	as depicted in~\cref{fig:HNF_LIST}. Even though we know that the quotient space consists of $|T_{\lA , \lA  M }|=|\operatorname{det}(M)|=10$ different elements, there is no apparent canonical list of these elements. Fortunately, a canonical ordering of the lattice points on a torus can be formulated using the Hermite or Smith normal form of $M$. %
	\begin{theorem}
		\label{cor:quotient_latticepointlist}
		Let $T_{\lA ,\lC }$ be arbitrary, i.e., $\lC  = \lA  M$ for some $M \in \m\mZ^{n \times n}$. \begin{itemize}
			\item Let $H \in\m\mZ^{n \times n}$ be the  HNF of $M$ with entries $H_{ij}$ (cf.~\cref{def:HNF}). Defining the index set $I=I_1 \times I_2 \times \ldots \times I_n$ by $I_\ell:=\{0,1,\ldots,H_{\ell\ell}-1\}$, we then obtain %
			\begin{align*}
				T^\se_{\lA ,\lC } & = \set{[x_j+\se]}{ x_j = \lA j,\ j \in I}
			\end{align*} %
			$\text{with } [x_j+\se] \neq [x_{j'}+\se] \ \Leftrightarrow\ j \neq j' \in I.$
			\item Let $S=U^{-1}MV \in \mZ^{n\times n}$ denote the Smith decomposition of $M$ with diagonal entries $s_{i}$ (cf.~\cref{def:SNF}) and unimodular matrices $U,V$. Defining the index set $\tilde{I}=\tilde{I}_1 \times \tilde{I}_2 \times \ldots \times \tilde{I}_n$ by $\tilde{I}_\ell:=\{0,1,\ldots,s_{\ell}-1\}$, we then obtain %
			\begin{align*}
				T^\se_{\lA ,\lC } & = \set{[x_j+\se]}{ x_j = \tilde{\lA} j,\ j \in \tilde{I}}\ 
			\end{align*} 
			$\text{with }\ [x_j+\se] \neq [x_{j'}+\se] \ \Leftrightarrow\ j \neq j' \in \tilde{I},$ where $\tilde{\lA} := \lA U$ denotes the altered lattice basis.
		\end{itemize}
		\begin{proof}
			Both statements are a direct consequence of the triangular or diagonal shape of the normal forms and \cref{thm:lattice_equiv_unimod}, i.e., lattices are not changed by unimodular column transformations. On the one hand we have
				$T^\se_{\lA ,\lC } = T^\se_{\lA ,\lA M } = T^\se_{\lA ,\lA H }.$ 
			The second statement follows from
			$\tilde{\lA} S = \lC V$ %
			and hence
				$T^\se_{\lA ,\lC } = T^\se_{\lA U , \lC V} = T^\se_{\tilde{\lA} ,\tilde{\lA} S }.$
		\end{proof}
	\end{theorem}
	In terms of the example of \cref{fig:HNF_LIST} we obtain:
	\begin{itemize}
		\item  The Hermite normal form $H$ of $M$ is given by 
		\[H= \mat[c]{h_{1} & h_{2}} = 
		\mat[c]{
			5 & 2 \\ 0 & 2},\ h_{1},h_{2} \in \mathbb{Z}^{2}
		.\] Thus, a unique list of all representatives of $T_{\lA ,\lA M}$ is given by 
		\begin{align*}
			T_{\lA ,\lA M}= T_{\lA ,\lA H} &= \set{[x]=[j_1 \la_1 + j_2 \la_2]}{(j_1,j_2) \in \{0,1,2,3,4\}\times\{0,1\} }. \\
		\end{align*}
		\item The Smith decomposition of $M$ is given by 
		\[S= \mat[c]{s_{1} & 0 \\ 0 & s_{2}} = 
		\mat[c]{
			1 & 0 \\ 0 & 10}
		=\mat[r]{
			1 & 0 \\ -4 & -1} M \mat[r]{
			-1 & -3 \\ 1 & 2}
		.\] Thus, another unique list of all representatives of $T_{\lA ,\lA M}$ is given by 
		$$T_{\lA ,\lA M}= T_{\tilde{\lA} , \tilde{\lA} S}=\set{[x]=[j_1 \tilde{\la}_1 + j_2 \tilde{\la}_2]}{(j_1,j_2) \in \{0\}\times\{0,1,\ldots,9\} },%
		$$  
		where $\tilde{\lA} = \mat{\tilde{\la}_1 & \tilde{\la}_2} = \lA \mat[r]{
			1 & 0 \\ -4 & -1}$.
	\end{itemize}
	All tori representations $T_{\lA ,\lA M}$, $T_{\lA ,\lA H}$
	and $T_{\tilde{\lA} , \tilde{\lA} S} $
	are depicted in~\cref{fig:HNF_LIST}.
	In the remainder we drop the bracket notation for reasons of readability.
	\begin{figure}
		\centering
		\resizebox{.8\textwidth}{!}{\includegraphics{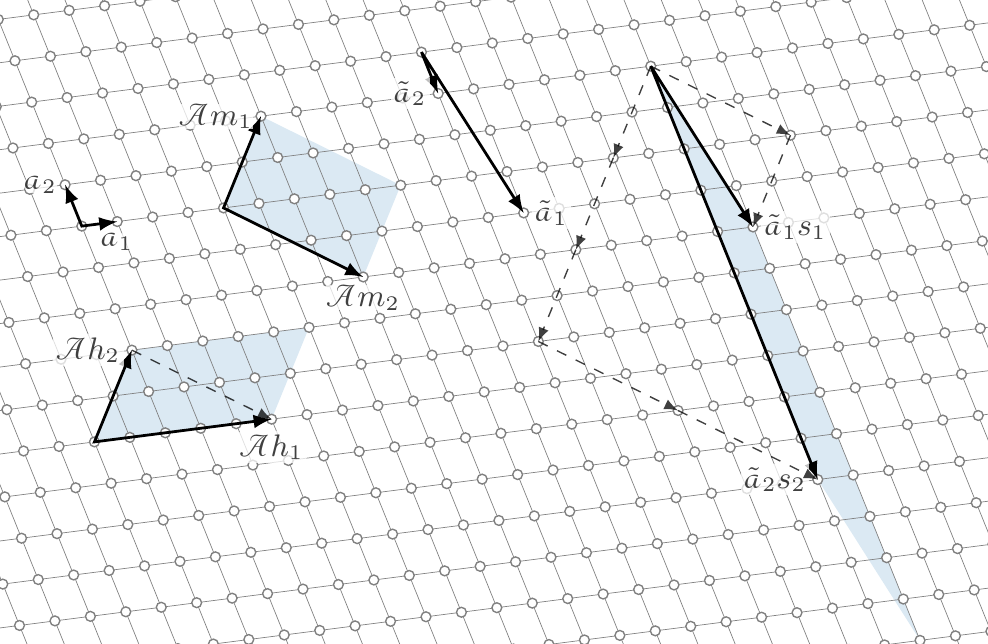}}
		\caption{The Hermite normal form $H$ and the Smith normal form $S$ of $M$ yield lattice bases which allow us to define a canonical lexicographic ordering of the lattice points of a crystal torus $T_{\lA,\lA M} = T_{\lA,\lA H} = T_{\tilde{\lA},\tilde{\lA} S}$.}
		\label{fig:HNF_LIST}
	\end{figure}

	\section{Operators on Crystals}\label{sec:operators}
	Now that the basic notation of the underlying structure is in place we introduce notation for value distributions and operators on these structures. For aforementioned reasons, we restrict ourselves to the finite dimensional setting by only considering quotient groups of lattices
	\begin{align*}
		T_{\lA,\lZ}^\se : = \bigslant{\mL^\se(\lA )}{\mL(\lZ )}
	\end{align*}
	with $\mL(\lZ  ) \subset \mL(\lA )$ being an arbitrary sublattice of $\mL(\lA )$. While it is significantly easier to be mathematically precise in this general finite setting than in an infinite setting it is also much closer to the targeted applications, namely finite dimensional approximations of PDEs and Hamiltonians. Eventually we only work with operators as part of numerical simulations, i.e., we face only a finite number of unknowns/lattice points anyway. The quotient group $T_{\lA,\lZ}^\se$, $\mL(\lZ) \subset \mL(\lA)$, precisely describes %
	such an arbitrarily large but finite torus. To shorten notation we use $T_{\lA}^\se$ instead of $T_{\lA,\lZ}^\se$ whenever we do not specify $\lZ$ explicitly.
	
	\begin{definition}
		A crystal-operator is a linear function
		\begin{align*}
			L : \fs( T_{\lA}^{\sed} ) \longrightarrow \fs( T_{\lA}^{\sec} ),
		\end{align*}
		where $T_{\lA}^{\sed}$ corresponds to the crystal of the domain and $T_{\lA}^{\sec}$ corresponds to the crystal of the codomain. The function spaces are defined by
		\begin{align*}
			\fs( T_{\lA}^{{\sel}} ) = \set{f=(f_1,\ldots,f_{|\sel|}) }{T_{\lA}\longrightarrow \mC^{|{{\sel}}|}}.
		\end{align*}
		A value $f_j(x)$, $x\in\mL(\lA)$, corresponds to a value at the position $x + \se_j$.
		The function space is equipped with the scalar product
		$$\ip[]{f}{g} := \frac{1}{|T_{\lA ,\lZ }|} \sum_{x\in T_{\lA ,\lZ }} \ip{f(x)}{g(x)},$$
		where $\ip{f(x)}{g(x)} := \sum_{\ell=1}^{|\sel|} f_\ell(x) \overline{g_\ell(x)}$ denotes the Euclidean scalar product on $\mC^{|\sel|}$.
	\end{definition}
	
	In the context of LFA we are interested in operators which can be represented in (block) stencil notation. That is, translationally invariant operators that can be written as multiplication operators. As these two properties are in fact equivalent (cf.~\cite[Theorem 3.16]{nla.cat-vn1897929}), we can connect those operators to the notation of crystal structures.
	
	\begin{theorem}
		Let $L : \fs( T_{\lA}^{\sed} ) \longrightarrow \fs( T_{\lA}^{\sec} )$ be a crystal operator. The following statements are equivalent.
		\begin{enumerate}
			\item $L$ is a multiplication operator, i.e., there exist matrices $m_L^{(y)} \in \mC^{|{\sec}|\times |{\sed}|}$ such that for each $x\in T_\lA$ and $f\in \fs( T_{\lA}^{\sed} ) $ we have
			\begin{align*}
				(Lf)(x) = \sum_{y \in T_\lA} m_L^{(y)} f(x+y).
			\end{align*}
			\item $L$ is $(\lA)$-translationally invariant, i.e.,
			$$\lop\trop_\la  - \trop_\la \lop = 0 \quad \text{for all (primitive) vectors } \la \in \mL(\lA  ), $$
			where the translation operator is defined by $(\trop_\la f)(x) = f(x+\la)$.
		\end{enumerate}
	\end{theorem}
	
	For the analysis of such operators the concept of the dual lattice comes in handy as already considered in similar form in~\cite{GaspGracLisb2009,ZhouFult}.
	\begin{definition}\label{def:dual_lattice}
		Let $\mL(\lA )$ be a lattice. Its \emph{dual lattice} $\mL(\lB)=\mL(\lA )^*$ is the set
		\begin{align*}
			\mL(\lA )^*:=\set{k \in \mR^n}{ \ip{k}{x} \in \mZ \text{ for all } x\in\mL}.
		\end{align*}
		A lattice basis of the dual lattice is given by $\lB = \lA^{-T}$. The elements of $\mL(\lA )^*$ may also be referred to as \emph{wave vectors}.
	\end{definition}
	
	In addition to the dual space we introduce an orthonormal basis of wave functions that are compatible with the crystal structure introduced in~\cref{sec:lattices_crystals}. This basis is an extension of the basis used in~\cite{Kuo1989} to analyze the red-black Gauss-Seidel relaxation. Furthermore, similar basis functions have been used in the context of LFA, e.g. in~\cite{BoonLentVand2008, 1803.08864, RittichBolten2018, 1902.10248, Brown2018LOCALFA}.
	\begin{theorem}
		\label{thm:wavefunction_ONB_vec}
		An orthonormal basis for the function space  $\fs(T^\se_{\lA ,\lZ })$ %
		with a structure element $\se=(\se_1,\ldots,\se_m)$ is given by the wave functions %
		$e_{1,k},e_{2,k},\ldots,e_{m,k}$ defined by
		\begin{align*}
			\left(e_{\ell,k}(x)\right)_j := 
			\begin{cases}
				\cexp{k}{x} & \text{ if }j=\ell, %
				\\ 0 & \text{ else,} 
			\end{cases}  
		\end{align*}
		with $k \in T_{\lA ,\lZ }^*:= \bigslant{\mL(\lZ)^*}{\mL(\lA)^*} $.
	\end{theorem}
	\begin{proof}
		The statement follows by a straightforward, but lengthy calculation, by assuming without loss of generality that $\lA^{-1}\lZ$ is given in Smith normal form, making use of \cref{cor:quotient_latticepointlist} and the geometric sum formula.
	\end{proof}
	An illustration of a lattice torus $T_{\lA,\lZ}$ along with its dual $T_{\lA,\lZ}^*$ is given in~\cref{fig:TorusAndDual}.
		\begin{figure}
		\resizebox{\textwidth}{!}{
			\resizebox{!}{4cm}{
				\includegraphics{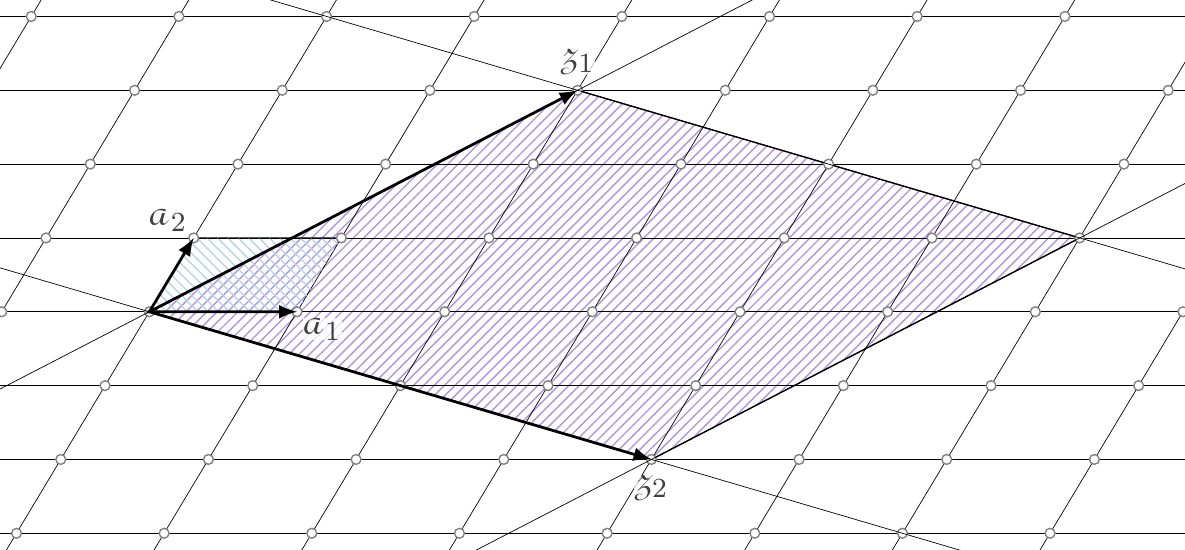}
			}
						\resizebox{!}{4cm}{
				\includegraphics{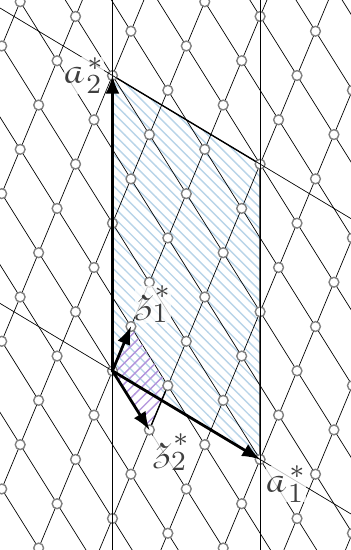}
			}
		}
		\caption{A lattice torus $T_{\lA,\lZ}$ (left) and its dual torus $T_{\lA,\lZ}^*$ (right). In here, the lattices bases are denoted by $\lA = \mat{\la_1 & \la_2},\ \lZ = \mat{\lz_1 & \lz_2}$ and $\lA^{-T} = \mat{\la_1^* & \la_2^*},\ \lZ^{-T} = \mat{\lz_1^* & \lz_2^*}$.\label{fig:TorusAndDual}}
	\end{figure}
	
	The orthonormal basis of~\cref{thm:wavefunction_ONB_vec} can be split into subsets with respect to the wave vector $k$, i.e.,
	$\fs(T^\se_{\lA ,\lZ }) = \cup_{k \in T_{\lA,\lZ}^*} \operatorname{span}(H_k) $ with
	\begin{equation}\label{eq:space-of-harmonics}
		H_k = \operatorname{span}\set{e_{\ell,k}}{\ell=1,\ldots,m}.
	\end{equation}
	
	\begin{theorem}\label{thm:L-invariance}
		Let $\lop : \fs(T_{\lA  ,\lZ }^\se) \rightarrow \fs(T_{\lA  ,\lZ }^\se)$ be a multiplication operator. %
		Then the subspaces $H_k$ of~\cref{eq:space-of-harmonics} are $L$-invariant, i.e., $L(H_k) \subseteq H_k$.
	\end{theorem}
	\begin{proof} Let $e_{k}(x)$ denote an arbitrary value distribution in $H_{k}$. That is, there exist $\alpha_{1},\ldots,\alpha_{m} \in \mC$ such that
		\begin{align*}
			e_k(x)=\sum_{\ell}\alpha_\ell e_{\ell,k} = 
			(\alpha_1 \cexp{k}{x},\ldots,\alpha_m\cexp{k}{x})^T \in \operatorname{span}(H_k).
		\end{align*}
		Then we obtain by direct calculation
		\begin{equation}\label{eq:symbol}
			(L e_k)(x)  = \sum_{y} m_L^{(y)} e_k(x+y) = \left(\sum_{y} m_L^{(y)} \cexp{k}{y}\right) e_k(x).
		\end{equation}
	\end{proof}
	
	Due to their $L$-invariance the subspaces $H_k$ are oftentimes referred to as \emph{spaces of harmonics}. %
	Thus we can easily represent any $\lA$-translationally operator via its \emph{symbols}, which are formally defined as follows.
	\begin{definition}
		\label{def:crystal_op_symbol}
		Let $L:\fs(T_{\lA  ,\lZ }^\se) \rightarrow \fs(T_{\lA  ,\lZ }^\st)$ be a multiplication operator with
		\begin{align*}
			(Lf)(x) = \displaystyle\sum_{y \in T_{\lA  ,\lZ }} m_L^{(y)} f(x+y),\  m_L^{(y)} \in \mC^{|\st|\times |\se|}.
		\end{align*}
		We define the \emph{symbol} of $L$ according to~\cref{eq:symbol} by
		\begin{align*}
			L_k:=\sum_{y \in T_{\lA  ,\lZ }} m_L^{(y)} m_k^{(y)}
			\text{ with }
			m_k^{(y)} :=  \cexp{k}{y}. %
		\end{align*}
	\end{definition}
	In case $\se = \st$ the spectrum of $L$ can then be extracted from its symbols $L_{k}$.
	
	\begin{theorem}\label{cor:eigenvalues_crystal_operator}
		Let $L:\fs(T_{\lA  ,\lZ }^\se) \rightarrow \fs(T_{\lA  ,\lZ }^\se)$ be a multiplication operator with
		\begin{align*}
			(Lf)(x) = \sum_{y \in T_{\lA  ,\lZ }} m_L^{(y)} f(x+y),\ m_L^{(y)} \in \mC^{|\se|\times |\se|}.
		\end{align*}
		Then $\spec(L) = \cup_{k\in T_{\lA,\lZ}^*} \spec(L_k)$.
	\end{theorem}
	\begin{proof} Follows immediately due to the orthonormality of the basis $e_{\ell,k}$ (cf.~\cref{thm:wavefunction_ONB_vec}) and the $L$-invariance of the subspaces $H_{k}$ (cf.~\cref{thm:L-invariance}).
	\end{proof}
	\begin{remark}\label{rem:freq_sampling}
		The main purpose of $\lZ$, i.e., the set of primitive vectors that define an arbitrary sublattice $\mL(\lZ)$ of $\mL(\lA)$, is to simplify the theory developed in~\cref{sec:operators} by turning an infinite dimensional setting to an (arbitrarily large) finite one. Though there is another interpretation to it as well. In~\cref{cor:eigenvalues_crystal_operator} $\lZ$ explicitly specifies the resolution of the frequency space as seen in~\cref{fig:TorusAndDual}, i.e.,  $$T_{\lA,\lZ}^* = \mL(\lZ^{-T}) \cap \PT(\lA^{-T}),$$
		where the spectrum of the multiplication operator is sampled. Due to the reciprocal nature of the dual space, the larger $|\operatorname{det}(\lZ)|$ is, the finer the resolution becomes.%
	\end{remark}
	
	With these tools at hand we are able to fully analyze a single multiplication operator, but typically we are interested in analyzing a composition of several operators using LFA. As long as the corresponding domains and codomains of these operators are compatible we can use the rules of computation given in \cref{sec:rules_of_comp} on the level of multiplication operators and/or on the level of the corresponding symbols. While computing the sum, a product and taking the transpose can easily be done on both levels, taking the (pseudo-)inverse is simple only on the level of symbols. The (pseudo-)inverse of a multiplication operator may have an arbitrarily large number\footnote{Bounded by the number of lattice points on the (arbitrarily large) torus.} of multipliers $m_{L^{-1}}^{(y)} \neq 0$ and thus there is no simple rule to compute it. In case a pseudo-inverse has to be used in the following we opt to employ the Moore-Penrose pseudo-inverse, which we denote by $S^{\dagger}$ for a multiplication operator $S$.
	
	In our framework, which tries to automate as much of the LFA as possible, we would like to allow for user friendly descriptions of all occurring operators. That is, it should be possible to describe operators in terms of their individual translational invariance and ordering of the structure element without having to worry about compatibility issues with other operators on the input level of the analysis. The process how to automatically make crystal representations of operators compatible is explained in detail in~\cref{sec:crystal_reps_and_isos}. Before diving into the gritty details of this automation process we would like to illustrate the developments made so far with an example in order to convey the introduced notation.

	\section{An example}\label{sec:example}
	Consider the red-black Gauss-Seidel method applied to the discretized Laplacian on the unit square with periodic boundary conditions. The most fundamental representation of the discretized unit square with periodic boundary conditions is given by the torus $T_{\lA,\lZ}$ with
	\begin{align*}
		\lZ=\mat[c]{1 & 0 \\ 0 & 1}, \lA=\mat[c]{\la_1 & \la_2}=\frac{1}{h} \mat[c]{1 & 0 \\ 0 & 1}\text{ for some }h=\frac{1}{N}, N\in\mN.
	\end{align*}
	Then, the discretized Laplacian $L_h : \fs( T_{\lA,\lZ}^{(0)} ) \longrightarrow \fs( T_{\lA,\lZ}^{(0)} )$ using finite differences is given by
	$(L_h f)(x) := \sum_{y\in\mL(\lA)} m_{L_h}^{(y)} f(x+y)$ with non-zero multipliers:
	\begin{center}
		\begin{tikzpicture}
			\pgfmathsetmacro{\hshift}{1.25}
			\pgfmathsetmacro{\vshift}{1}
			
			\node (c) at (0,0) {$=$}; %
			\node[left=-.1cm of c, anchor=east] {$m_{L_h}^{(0)}$};
			\node[right=-.1cm of c, anchor=west] {$\frac{4}{h^2}$};
			
			\node (c) at (0,\vshift) {$=$}; %
			\node[left=-.1cm of c, anchor=east] {$m_{L_h}^{(\la_2)}$};
			\node[right=-.1cm of c, anchor=west] {$-\frac{1}{h^2}$};
			
			\node (c) at (0,-\vshift) {$=$}; %
			\node[left=-.1cm of c, anchor=east] {$m_{L_h}^{(-\la_2)}$};
			\node[right=-.1cm of c, anchor=west] {$-\frac{1}{h^2}$};
			
			\node (c) at (2.2*\hshift,0) {$=$}; %
			\node[left=-.1cm of c, anchor=east] {$m_{L_h}^{(\la_1)}$};
			\node[right=-.1cm of c, anchor=west] {$-\frac{1}{h^2}$};
			
			\node (c) at (-1.9*\hshift,0) {$=$}; %
			\node[left=-.1cm of c, anchor=east] {$m_{L_h}^{(-\la_1)}$};
			\node[right=-.1cm of c, anchor=west] {$-\frac{1}{h^2}$};
		\end{tikzpicture}
	\end{center}

	The error propagator $G$ of the red-black Gauss-Seidel method can be written with respect to the crystal $T_{\lA,\lZ}^{(0)}$ via
	\begin{align*}
		G=(I-S_b^\dagger L_h)(I-S_r^\dagger L_h) : \fs( T_{\lA,\lZ}^{(0)} ) \longrightarrow \fs( T_{\lA,\lZ}^{(0)} )
	\end{align*}
	with
	\begin{align*}
		(S_r f)(x) :=
		\begin{cases}
			\frac{4}{h^2} f(x) & x\in X_\text{red}   \\
			0                  & x\in X_\text{black}
		\end{cases}
		\text{ and } (S_b f)(x) :=
		\begin{cases}
			0                  & x\in X_\text{red}   \\
			\frac{4}{h^2} f(x) & x\in X_\text{black}
		\end{cases}
		,
	\end{align*}
	where $X_\text{red}$ and $X_\text{black}$ correspond to the red and black unknowns of the torus $T_{\lA,\lZ}^{(0)}$ as illustrated in \cref{fig:redblack_laplacian}, b). In order to analyze this composite operator in our framework, we write all occurring operators as multiplication operators. It is now important that the red-black splitting of $T_{\lA,\lZ}^{(0)} = X_{\text{red}} \cup X_{\text{black}}$ implies the crystal $ T_{\lC,\lZ}^\se$  with 
	\begin{align*}
		\lc_1 = \la_1 + \la_2, \text{ } \lc_2 = \la_1 - \la_2, \ \se := (0, \la_1)  = \mL(\lA) \cap \PT(\lC)
	\end{align*} such that $T_{\lC,\lZ}^{(0)} =  X_{\text{red}}$,  $T_{\lC,\lZ}^{(\la_1)} = X_{\text{black}}$  and $T_{\lA,\lZ}^{(0)} = X_{\text{red}} \cup X_{\text{black}} \cong T_{\lC,\lZ}^\se$ (cf.~\cref{fig:redblack_laplacian}, b)).
	With respect to this crystal the operators $S_r$ and $S_b$ can be written as multiplication operators
	$\hat{S}_r, \hat{S}_b : \fs( T_{\lC,\lZ}^{\se} ) \longrightarrow \fs( T_{\lC,\lZ}^{\se} )$
	with
	\begin{align*}
		(\hat{S}_r f)(x) = \mat[c]{\frac{4}{h^2} & 0 \\ 0 & 0} f(x) \text{ and } (\hat{S}_b f)(x) = \mat[c]{0 & 0 \\ 0 & \frac{4}{h^2}} f(x).
	\end{align*}
	
	We now have described all individual operators of $G$, each one defined with respect to its own (minimal) translational invariance, but the domains and codomains are not identical, i.e., $\lA \neq \lC$, such that we cannot directly use the computation rules described in \cref{sec:rules_of_comp}. In order to carry on with the analysis we have to rewrite the operators with respect to a common crystal structure. In this example we construct this structure by hand, but this process can be automated as explained in~\cref{sec:crystal_reps_and_isos}.
	
	As the crystal $T_{\lC,\lZ}^\se$ is yet another representation of $T_{\lA,\lZ}^{(0)}$, the operator $L_h$ can be rewritten with respect to this crystal (cf.~\cref{fig:redblack_laplacian}) as
	$ \hat{L}_h: \fs( T_{\lC,\lZ}^{\se} ) \longrightarrow \fs( T_{\lC,\lZ}^{\se} )$ with $(\hat{L}_h f)(x) := \sum_{y\in\mL(\lC)} m_{\hat{L}_h}^{(y)} f(x+y)$ and non-zero multipliers:
	\begin{center}
		\resizebox{.99\linewidth}{!}{
			\begin{tikzpicture}
				
				\pgfmathsetmacro{\hshift}{2.2}
				\pgfmathsetmacro{\vshift}{1}
				
				\node (c) at (0,0) {$=$}; %
				\node[left=-.1cm of c, anchor=east] {$m_{\hat{L}_h}^{(0)}$};
				\node[right=-.1cm of c, anchor=west] {$\frac{1}{h^2}\mat[r]{ 4 & -1 \\ -1 & 4}$};
				
				\node (c) at (-\hshift,\vshift) {$=$}; %
				\node[left=-.1cm of c, anchor=east] {$m_{\hat{L}_h}^{(-\lc_2)}$};
				\node[right=-.1cm of c, anchor=west] {$\frac{1}{h^2} \mat[r]{ 0 & 0 \\ -1 & 0}$};
				
				\node (c) at (\hshift,\vshift) {$=$}; %
				\node[left=-.1cm of c, anchor=east] {$m_{\hat{L}_h}^{(\lc_1)}$};
				\node[right=-.1cm of c, anchor=west] {$\frac{1}{h^2}  \mat[r]{ 0 & -1 \\ 0 & 0}$};
				
				\node (c) at (-\hshift,-\vshift) {$=$}; %
				\node[left=-.1cm of c, anchor=east] {$m_{\hat{L}_h}^{(-\lc_1)}$};
				\node[right=-.1cm of c, anchor=west] {$\frac{1}{h^2} \mat[r]{ 0 & 0 \\ -1 & 0}$};
				
				\node (c) at (\hshift,-\vshift) {$=$}; %
				\node[left=-.1cm of c, anchor=east] {$m_{\hat{L}_h}^{(\lc_2)}$};
				\node[right=-.1cm of c, anchor=west] {$\frac{1}{h^2}  \mat[r]{ 0 & -1 \\ 0 & 0}$};
				
				\node (c) at (2.2*\hshift,0) {$=$}; %
				\node[left=-.1cm of c, anchor=east] {$m_{\hat{L}_h}^{(\lc_1+\lc_2)}$};
				\node[right=-.1cm of c, anchor=west] {$\frac{1}{h^2}  \mat[r]{ 0 & -1 \\ 0 & 0}$};
				
				\node (c) at (-1.9*\hshift,0) {$=$}; %
				\node[left=-.1cm of c, anchor=east] {$m_{\hat{L}_h}^{-(\lc_1+\lc_2)}$};
				\node[right=-.1cm of c, anchor=west] {$\frac{1}{h^2}  \mat[r]{ 0 & 0 \\ -1 & 0}$};
			\end{tikzpicture}%
		}
	\end{center}
	\begin{figure}
		\TikzFigsABresize[a)\ $L_h: \fs(T_{\lA,\lZ }^{(0)}) \rightarrow \fs(T_{\lA,\lZ }^{(0)})$][ b)\ $\hat{L}_h: \fs(T_{\lC,\lZ }^{\se}) \rightarrow \fs(T_{\lC,\lZ }^{\se})$]{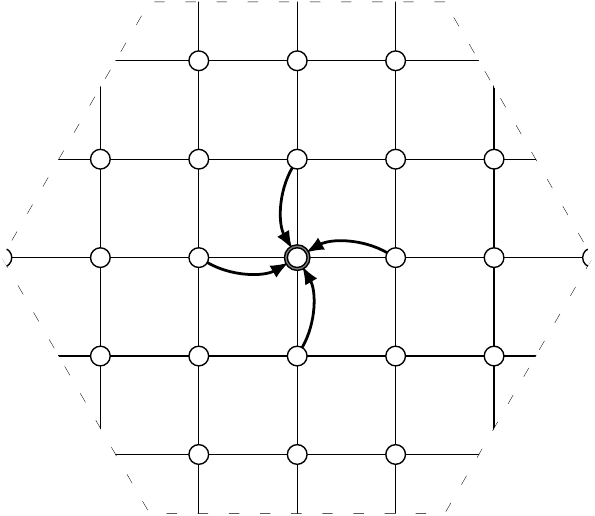}{red-black-stencil01}
		\caption[The discretized Laplace operator with respect to two different crystal representations.]{The discretized Laplace operator with respect to two different crystal representations. In a) the operator is illustrated with respect to the primitive vectors $\la_1, \la_2$ and in b) with respect to $\lc_1,\lc_2$.}\label{fig:redblack_laplacian}
	\end{figure}
	
	Thus, the spectrum of the error propagator of the red-black Gauss-Seidel method applied to the Laplacian
	\begin{equation}\label{eq:RBGEP}
		\hat{G} = (I-\hat{S}_b^\dagger \hat{L}_h)(I-\hat{S}_r^\dagger \hat{L}_h)
	\end{equation}
	can now be obtained elementwise for each fixed $k \in T_{\lC,\lZ}^* = \PT(\lC^{-T}) \cap \mL(\lZ^{-T})$ by first computing the individual \emph{symbols} $I_k, (\hat{S}_r)_k, (\hat{S}_b)_k$ and  $(\hat{L}_h)_k$ and assembling the symbols $\hat{G}_k$ according to~\cref{eq:RBGEP} and the rules in~\cref{sec:rules_of_comp}, followed by the computation of the eigenvalues of the matrices $\hat{G}_k$. %
	The resulting spectral information of the discretized Laplace operator $\hat{L}_h$ and the error propagator $\hat{G}$ is illustrated in~\cref{fig:redblack_spectrum}, where it is sampled on the dual lattice $T_{\lC,\lZ}^*$. Note, that one naturally obtains two eigenvalues per sampled wave vector $k$. In case of the spectrum of the red-black Gauss-Seidel error propagator, one of the two eigenvalues is equal to zero for all wave vectors $k$.%
	
	\begin{figure}
		\TikzFigsABresize[a)\ $\hat{L}_h: \fs(T^\se_{\lC,\lZ }) \rightarrow \fs(T^\se_{\lC,\lZ })$][ b)\ $\hat{G}_h: \fs(T_{\lC,\lZ }^{\se}) \rightarrow \fs(T_{\lC,\lZ }^{\se})$]{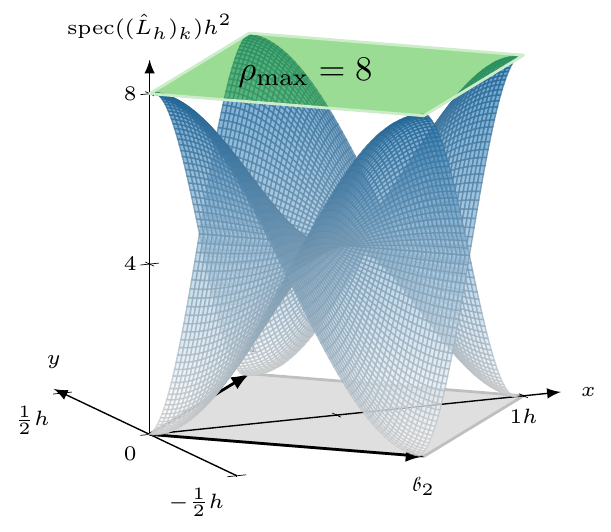}{LFA-laplacian-5p-rb-jacobi-0}
		
		\caption{The spectra of a) the discretized Laplace operator and b) the red-black Gauss-Seidel smoother with respect to the red-black crystal structure. In here, $\lb_1$ and $\lb_2$ denote the columns of $\lC^{-T}$.}\label{fig:redblack_spectrum}
	\end{figure}

	\section{Crystal representations and natural isomorphisms}\label{sec:crystal_reps_and_isos}
	In general, we are given several multiplication operators which make up the error propagator of an iterative method, each defined with respect to its own (minimal) translational invariance. In order to analyze the method we thus need to find a common denominator, i.e., a lattice basis corresponding to the collective translational invariance, and rewrite the operators accordingly. %
	The following theorem yields a set of primitive vectors of such a collective translational invariance for two arbitrary lattices, if it exists. %

	\begin{theorem}
		\label{thm:lcm_mat}
		Given two $n$-dimensional lattices $\mL(\lA)$, $\mL(\lB)$. If there exists an $r\in \mZ$, such that $M = r \lA^{-1} \lB \in \mZ^{n\times n}$,
		then there is a lattice $\mL(\lC)$ with
		$\mL(\lC) \subset \mL(\lA)$ and $\mL(\lC) \subset \mL(\lB)$ with $|\det(\lC)|$ as small as possible. A lattice basis of $\mL(\lC)$ is given by 
		\begin{align*}
			\lC = \lB T^{-1} N_\lB, \quad
		\end{align*}
		where $N_\lB$ is a diagonal matrix with $(N_\lB)_{i,i} := r \cdot \operatorname{gcd}(r,s_{i})^{-1}$, 
		where $S= V^{-1} M T^{-1} = \operatorname{diag}{(s_{1},\ldots,s_{n})}$ is the SNF of $M$ (cf.~\cref{def:SNF}) and $\operatorname{gcd}(r,s_{i})$ denotes the greatest common divisor of $r$ and $s_i$.  Consequently, we write $\mL(\lC) = \operatorname{lcm}( \mL(\lA), \mL(\lB))$ and call it the \emph{least common multiple} of $\mL(\lA)$ and $\mL(\lB)$.
		\begin{proof}
			Due to \cref{lem:sublattice} it is sufficient to find integral matrices $N_\lA, N_\lB$, such that
			\begin{align*}
				\mL(\lC) = \mL(\lA N_\lA) = \mL(\lB N_\lB)
			\end{align*}
			with $|\det(N_\lA)|$ and $|\det(N_\lB)|$ as small as possible. Using \cref{thm:lattice_equiv_unimod}, i.e.,
			$
			\mL(\lB) = \mL(\lB V_1), \quad \mL(\lC) = \mL(\lC U_1),
			$
			for any unimodular matrices $U_1,V_1$, we can assume the equality
			\begin{align*}
				\lA N_\lA = \lB U N_\lB %
			\end{align*}
			for any unimodular $U$ and $N_\lB$ in Hermite normal form (cf.~\cref{thm:integer_normalforms}). Plugging in the Smith decomposition $VST$ of $M=r \lA^{-1} \lB$ and defining $U:=T^{-1}$, we find 
			\begin{align*}
				\begin{array}{rcccl}
					N_\lA & = & \lA^{-1} \lB T^{-1} N_\lB &= &  \frac{1}{r} VS N_\lB.     
				\end{array}
			\end{align*}
			Both matrices
			\begin{align*}
				N_\lB = \mat{(N_\lB)_{1,1} & \ldots & (N_\lB)_{1,n} \\
					& \ddots & \vdots \\
					& & (N_\lB)_{n,n} } \text{ and } \frac{1}{r} S N_\lB = \mat{\frac{s_{1}}{r}(N_\lB)_{1,1} & \ldots & (N_\lB)_{1,n}\frac{s_{1}}{r} \\
					& \ddots & \vdots \\
					& & (N_\lB)_{n,n}\frac{s_{n}}{r} }
			\end{align*}
			have to be integral with $$|\det(N_\lB)| = |\prod_{i=1}^n (N_\lB)_{i,i}| \text{ and }|\det(\frac{1}{r} S N_\lB)| = |\prod_{i=1}^n (N_\lB)_{i,i} \frac{s_{i}}{r} |$$ as small as possible.
			It can easily be verified that 
			\begin{align*}
				(N_\lB)_{i,i} := \frac{r}{\operatorname{gcd}(r,s_{i})}
			\end{align*}
			is the optimal choice for the diagonal entries. With this choice, the off-diagonal entries $(N_\lB)_{i,j}$ have to be integral multiples of $(N_\lB)_{i,i}$. Due to the fact that $N_\lB$ is in Hermite normal form, the off-diagonal entries are zero.
		\end{proof}
	\end{theorem}

	We now study different representations of the same crystal structure in order to derive a general way to rewrite a multiplication operator with respect to some coarser crystal structure corresponding to a sublattice, as has been done manually in \cref{sec:example} for the discretized Laplacian.
	
	\begin{theorem}[Rewriting a crystal with respect to a sublattice]
		\label{thm:crystal_congruence}
		Let $\mL^\se(\lA)$ be a crystal and $\mL(\lC) \subset \mL(\lA)$ a sublattice. Denoting $T_{\lA,\lC}=\{\st_1,\ldots,\st_p\}$,%
		\footnote{Recall that a unique list of representatives $T_{\lA,\lC} = \{\st_1,\ldots,\st_p\}$ can be found via \cref{cor:quotient_latticepointlist}.} 
		the set \begin{align*}
			\pc(\lC ):=\set{x+\delta\in\mR^n}{\delta\in\PT(\lA ), x \in \{\st_1,\ldots,\st_p\}}
		\end{align*} defines a primitive cell of $\mL(\lC)$,
		and the tuple
		\begin{align*}
			\seu = (\st_1+\se_1,\ldots , \st_1+\se_m, \st_2+\se_1,\ldots,\st_p+\se_m)\in\pc(\lC )^{p\cdot m}
		\end{align*}
		defines a structure element of $\mL^\seu(\lC)$ such that
		$\mL^\seu(\lC) \cong \mL^\se(\lA),$
		meant as a one-to-$p$ correspondence.
	\end{theorem}
	\begin{proof}
		Without loss of generality we may assume $\st_{j} \in \PT(\lC)$, $j = 1,\ldots,p$. Then, each element in $\mL^\se(\lA )$ can be written as
		\begin{align*}
			z=(\lA x+\se_1,\ldots,\lA x+\se_m)
		\end{align*}
		and there is a unique $y$, such that $\lA x=\lC y=\lC \lfloor y \rfloor + \lC (y - \lfloor y \rfloor)$ with $\lC \lfloor y \rfloor \in \mL(\lC )$ and $\lC (y - \lfloor y \rfloor)=\st_j \in \PT(\lC )\cap\mL(\lA )$. Thus, we find $z$ as a unique part of the element 
		\begin{align*}
			\lC \lfloor y \rfloor +{\seu}=(\ldots,\lC \lfloor y \rfloor  + \st_j + \se,\ldots)=(\ldots,z,\ldots).
		\end{align*}
		This argument works in the other direction in the same way.
	\end{proof}
	\begin{remark}\label{rem:natural_isomorph}
		Note, that $\seu$, as defined in~\cref{thm:crystal_congruence}, is an explicit representation of $T^\se_{\lA,\lC}$, thus using this explicit representation~\cref{thm:crystal_congruence} implies a congruence of the function spaces
	\begin{align*}
		\fs(T_{\lA ,\lZ }^\se) \cong \fs(T_{\lC ,\lZ }^{T^\se_{\lA ,\lC }})
	\end{align*}
	given by the natural isomorphism $\eta: \fs(T_{\lA ,\lZ }^\se) \rightarrow \fs(T_{\lC ,\lZ }^{T^\se_{\lA ,\lC }})$,
	\begin{align*}
		\fs(T_{\lA ,\lZ }^\se) \ni f(\cdot) \mapsto (\eta f )(\cdot)= (f(\cdot+\st_1),\ldots, f(\cdot+\st_p)) \in \fs(T_{\lC ,\lZ }^{T^\se_{\lA ,\lC }}),
	\end{align*}
	as $(f)$ and $(\eta f)$ describe the same value distribution on the crystal. This congruence in turn implies that the coarsest possible crystal interpretation, i.e., $\seu \cong T^\se_{\lA ,\lZ }$, is simply the complex coordinate space
	$$\mC^{mn} = \fs(T_{\lZ ,\lZ }^{T^\se_{\lA ,\lZ }})  \cong \fs(T^\se_{\lA ,\lZ })$$
	with $n:=|T_{\lA ,\lZ }|$. The scalar product on $\fs(T^\se_{\lA ,\lZ })$ then corresponds to the Euclidean scalar product on $\mC^{mn}$ up to a factor of $n$. %
	\end{remark}

	Using the natural isomorphism of function spaces in~\cref{rem:natural_isomorph}, we can derive the transformations of multiplication operators when coarsening the underlying crystal representation corresponding to a sublattice.
	
	\begin{theorem}[Rewriting a multiplication operator with respect to a sublattice]\label{thm:operator_sublattice_coarsening}
		Consider crystals $\mL^\sed(\lA)$, $\mL^\sec(\lA)$, a sublattice $\mL(\lC) \subset \mL(\lA)$ and a multiplication operator
		\begin{align*}
			L: \fs(T_\lA ^{\sed}) \rightarrow \fs(T_\lA ^{\sec}), \quad (\lop f)(x):=\sum_{y\in T_{\lA }} m_L^{(y)} f(x+y), \quad m_L^{(y)} \in \mC ^{|\sec|\times |\sed|}.
		\end{align*}
		Then, using $T_{\lA,\lC}=\{\st_1,\ldots,\st_p\}$,
		the multiplication operator 
		\begin{align*}
			G:\fs(T_\lC ^{\hat{\sed}}) \rightarrow  \fs(T_\lC ^{\hat{\sec}}), \quad (G g)(x)=\sum_{y\in T_\lC } m_G^{(y)} g(x+y), \quad m_G^{(y)} \in \mC ^{p |\sec|\times p|\sed|},
		\end{align*}
		with block matrices
		$(m_G^{(y)})_{i,k} := m_L^{(y-\st_i+\st_k)} \in \mC ^{|\sec|\times |\sed|}$
		fulfills the commutative diagram:
		\[
		\begin{tikzcd}
			\fs(T_\lA ^{\sed}) \arrow[r, "\lop"] \arrow[d, "\eta^\sed"]
			& \fs(T_\lA ^{\sec}) \arrow[d, "\eta^\sec" ] \\
			\fs(T_\lC ^{\hat{\sed}}) \arrow[r, "G" ]
			& \fs(T_\lC ^{\hat{\sec}}).
		\end{tikzcd}
		\]
		Here, the mappings for $\se \in \{\sed,\sec\}$,
		\begin{align*}
			\eta^\se: \fs(T_\lA ^{\se}) \rightarrow \fs(T_\lC ^{\hat{\se}}), \quad f(\cdot)\mapsto (f(\cdot +\st_1),\ldots,f(\cdot+\st_p)),
		\end{align*}
		denote the natural isomorphisms between the congruent crystal representations.%
	\end{theorem}
	\begin{proof}
		A straightforward calculation for each block-row $i$ yields
		\begin{align*}
			\begin{array}{rcl}
				[(\eta^\sec L f)(x)]_i & = & (Lf)(x+\st_i)                                                            %
				\ =\ \displaystyle\sum\limits_{k=1}^{p} \displaystyle\sum\limits_{y \in T_\lC } m_L^{(y+\st_k)} f(x+y+\st_i+\st_k) \\[\medskipamount]%
				& = &  \displaystyle\sum\limits_{k=1}^{p} \displaystyle\sum\limits_{y \in T_\lC } m_L^{(y-\st_i+\st_k)} f(x+y+\st_k) %
				\ =\   \displaystyle\sum\limits_{y \in T_\lC } \displaystyle\sum\limits_{k=1}^{p} (m_G^{(y)})_{i,k} f(x+y+\st_k)      \\[\bigskipamount]%
				& = & [G (f(x+\st_1),\ldots,f(x+\st_p))]_i                                    %
				\ =\ [(G \eta^\sed f)(x)]_i.
			\end{array}
		\end{align*}
	\end{proof}
	
	Using~\cref{thm:operator_sublattice_coarsening} we now know how to rewrite multiple multiplication operators with respect to some common crystal structure with a coarser translational invariance. %
	Due to the fact that we do not make any assumption on the initial representation of the crystal structures, the resulting structure elements of~\cref{cor:quotient_latticepointlist} might differ in their orderings and might contain shifts with respect to the common shift invariance. To automatically remove these differences and determine the corresponding transformations of the associated multiplication operators we first define the notion of congruent structure elements. %
	
	\begin{definition}\label{def:crystaltori_congruent}
		Two crystal tori $T_\lA ^\se \cong T_\lA ^\st$, $\lA \in \mR^n$ are \emph{congruent} with respect to $\mL(\lA)$ if
		the structure elements are of the same size, i.e., 
		$|\se| = |\st| = m$, %
		and there is a permutation $\pi:\{1,\ldots,m\} \rightarrow \{1,\ldots,m\}$ as well as shifts $y_j \in \mL(\lA )$, such that
		\begin{align*}
			\se_j =  y_j + \st_{\pi(j)} %
		\end{align*}
	\end{definition}
	
	In order to introduce a unique representation for the sake of automation, we introduce the following \emph{normal form} and the required transformations to transfer any operator to this form. %
	\begin{definition}\label{def:crystal_op_and_normalform}
		Let
		$L: \fs(T_\lA ^{\sed}) \rightarrow \fs(T_\lA ^{\sec})$ 
		be a multiplication operator.
		We say $L$ is in \emph{normal form} if
		\begin{itemize}
			\item the coordinates of the structure elements are found in the primitive cell spanned by the primitive vectors, i.e., $\sed_i,\sec_j \in \PT(\lA) = \lA[0,1)^n \text{\ for each\ } i,j,$
			\item the structure elements $\sed$ and $\sec$ are sorted lexicographically.\footnote{In case $\sed_i=\sed_j$ or $\sec_i = \sec_j$ for any $i\neq j$ a consistent ordering of $i,j$ has to be defined a priori.}
		\end{itemize}
	\end{definition}
	
	We now derive the implications of \cref{def:crystaltori_congruent} for multiplication operators when the structure element is element wise shifted or permuted. We do so in two steps, \cref{thm:shifted_structure_element} and \cref{thm:permuted_structure_element}. First, we show that a shift of an entry of the structure element in the codomain or domain results in a modification of the corresponding row or column of the non-zero multipliers, respectively. %
	
	\begin{theorem}[Shifted structure elements]\label{thm:shifted_structure_element}
		Consider the two multiplication operators
		$ L: \fs(T_\lA ^{\se}) \rightarrow \fs(T_\lA ^{\st})$ and $G: \fs(T_\lA ^{\st}) \rightarrow \fs(T_\lA ^{\seu})$
		defined by
		\begin{align*}
			(Lf)(x) = \sum_{y \in T_{\lA }} m_L^{(y)} f(x+y), \quad m_L^{(y)} \in \m\lC ^{|\st|\times |\se|},  \\
			(Gg)(x) = \sum_{y \in T_{\lA }} m_G^{(y)} g(x+y), \quad m_G^{(y)} \in \m\lC ^{|\seu| \times |\st|}.
		\end{align*}
		Further let $\hat{\st}$ be a structure element which is obtained from $\st$ when shifted element-wise along $\mL(\lA)$, i.e.,
		\begin{align*}
			\st & =  (\st_1,\ldots,\st_m)    =  (\hat{\st}_1 + y_1,\ldots,\hat{\st}_m+y_m ) = \hat{\st} +(y_1,\ldots,y_m),
		\end{align*}
		where $y_1,\ldots,y_m \in\mL(\lA )$ and $m=|\st|$. Then, the operators $\hat{L}$ and $\hat{G}$ given by
		\begin{align*}
			\begin{array}{rcllcl}
				(\hat{L}f)(x) &=& \displaystyle\sum\limits_{y \in T_{\lA }} m_{\hat{L}}^{(y)} f(x+y), & m_{\hat{L}}^{(y)} \in \m\lC ^{|\st|\times |\se|} &\text{ with }& (m_{\hat{L}}^{(y)})_{i,j} := (m_{L}^{(y + y_i)})_{i,j},\\
				(\hat{G}f)(x) &=& \displaystyle\sum\limits_{y \in T_{\lA }} m_{\hat{G}}^{(y)} f(x+y), & m_{\hat{G}}^{(y)} \in \m\lC ^{|\seu|\times |\st|} &\text{ with }& (m_{\hat{G}}^{(y)})_{i,j} := (m_{G}^{(y - y_j)})_{i,j}
			\end{array}
		\end{align*}
		fulfill the commutative diagram:
		\[
		\begin{tikzcd}
			\fs(T_\lA ^\se) \arrow[r, "\lop"] \arrow[rd, "\hat{L}"]
			& \fs(T_\lA ^\st) \arrow[r, "G"] \arrow[d, "\trop" ]
			& \fs(T_\lA^\seu) %
			\\
			&\fs(T_\lA ^{\hat{\st}}) \arrow[ur,"\hat{G}"]
		\end{tikzcd}
		\]
	\end{theorem}
	\begin{proof}
		The natural isomorphism between the two corresponding function spaces is given by
		\begin{align*}
			\trop: \fs(T_\lA ^\st) \rightarrow \fs(T_\lA ^{\hat{\st}}),\  f=(f_1,\ldots,f_m)\mapsto (f_1(\cdot - y_1),\ldots,f_m(\cdot-y_m)) = \trop f
		\end{align*}
		as $f$ and $(\trop f)$ describe the same value distribution on the crystal.\footnote{The left-hand side of $[(\trop f)(x)]_i=f_i(x-y_i)$ corresponds to the value at position $x + {\hat{\st}}_i = (x-y_i)+\st_i$ which coincides with the position of the value of the right-hand side.}
		Again, a straightforward calculation yields
		\begin{align*}
			\begin{array}{rcl}
				[(\trop Lf)(x)]_i & =& [(Lf)(x-y_i)]_i                                                              %
				\ =\  \displaystyle\sum\limits_{y \in T_{\lA }} \displaystyle\sum\limits_{j=1}^{|\se|} (m_L^{(y)})_{i,j}  f_j(x-y_i+y)  \\ %
				& =&   \displaystyle\sum\limits_{y \in T_{\lA }} \displaystyle\sum\limits_{j=1}^{|\se|}  (m_L^{(y+y_i)})_{i,j}  f_j(x+y) %
				\ =\   [\displaystyle\sum\limits_{y \in T_{\lA }}  m_{\hat{L}}^{(y)} f(x+y)]_i.
			\end{array}
		\end{align*}
		Analogously we find $[( G \trop^{-1} g)(x)]_i = [\sum_{y \in T_{\lA }}  (m_{\hat{G}}^{(y)})  g(x+y)]_i$.
	\end{proof}
	
	Finally, we show that permutations of the entries of the structure element result in a transformation of the non-zero multipliers by corresponding permutation matrices. 
	
	\begin{theorem}[Permuted structure elements]\label{thm:permuted_structure_element}
		Consider the two multiplication operators
		$ L: \fs(T_\lA ^{\se}) \rightarrow \fs(T_\lA ^{\st})$ and $G: \fs(T_\lA ^{\st}) \rightarrow \fs(T_\lA ^{\seu})$
		defined by
		\begin{align*}
			(Lf)(x) = \sum_{y \in T_{\lA }} m_L^{(y)} f(x+y), \quad m_L^{(y)} \in \m\lC ^{|\st|\times |\se|},  \\
			(Gg)(x) = \sum_{y \in T_{\lA }} m_G^{(y)} g(x+y), \quad m_G^{(y)} \in \m\lC ^{|\seu| \times |\st|}.
		\end{align*}
		Let further $\hat{\st}$ be a structure element which is a permuted version of $\st$, i.e.,
		\begin{align*}
			\hat{\st} & =  (\hat{\st}_1,\ldots,\hat{\st}_m)  =  (\st_{\pi(1)},\ldots, \st_{\pi(m)}) = m_\pi \st
		\end{align*}
		where $m=|\st|$, $\pi:\{1,\ldots,m\} \rightarrow \{1,\ldots,m\}$ is a permutation and $m_\pi \in \{0,1\}^{m\times m}$ the corresponding permutation matrix. Then, the operators $\hat{L}$ and $\hat{G}$ given by
		\begin{align*}
			\begin{array}{rcllcl}
				(\hat{L}f)(x) &=& \displaystyle\sum\limits_{y \in T_{\lA }} m_{\hat{L}}^{(y)} f(x+y), & m_{\hat{L}}^{(y)} \in \m\lC ^{|\st|\times |\se|} &\text{ with }& m_{\hat{L}}^{(y)} := m_\pi m_{L}^{(y)},\\
				(\hat{G}f)(x) &=& \displaystyle\sum\limits_{y \in T_{\lA }} m_{\hat{G}}^{(y)} f(x+y), & m_{\hat{G}}^{(y)} \in \m\lC ^{|\seu|\times |\st|} &\text{ with }& m_{\hat{G}}^{(y)} := m_{G}^{(y)} m_\pi^{-1}
			\end{array}
		\end{align*}
		fulfill the commutative diagram:
		\[
		\begin{tikzcd}
			\fs(T_\lA ^\se) \arrow[r, "\lop"] \arrow[rd, "\hat{L}"]
			& \fs(T_\lA ^\st) \arrow[r, "G"] \arrow[d, "p" ]
			& \fs(T_\lA^\seu) %
			\\
			&\fs(T_\lA ^{\hat{\st}}) \arrow[ur,"\hat{G}"]
		\end{tikzcd}
		\]
	\end{theorem}
	\begin{proof}
		Due to the fact that the natural isomorphism $p: \fs(T_\lA ^\st) \rightarrow  \fs(T_\lA ^{\hat{\st}})$ is a multiplication operator defined by $(p f)(x) = m_\pi f(x)$ for all $x\in\mL(\lA)$, the statement is true due to the rules of computation in \cref{lem:calc_rules_mult_op}.
	\end{proof}
	
	\Cref{cor:quotient_latticepointlist} and \cref{thm:lcm_mat,thm:crystal_congruence,thm:operator_sublattice_coarsening,thm:shifted_structure_element,thm:permuted_structure_element} 
	allow for the automatic adjustment of crystal representations within the LFA. The corresponding detailed algorithms which make use of these results are given in \cref{sec:algorithms}.

	\section{Application}\label{sec:Application}
	Before demonstrating the application of aLFA to some selected examples let us briefly recapitulate the individual parts of the framework. The introduction of crystal structures in~\cref{sec:lattices_crystals} allow for a canonical description of translationally invariant operators introduced in~\cref{sec:operators}. Combined with the definition of the dual of a crystal structure in~\cref{def:dual_lattice} and a corresponding orthonormal basis in~\cref{thm:wavefunction_ONB_vec}, the symbol of any single multiplication operator that is translationally invariant with respect to an arbitrary lattice structure, can be expressed (cf.~\cref{def:crystal_op_symbol}). This combination of choice of basis functions and the explicit connection of operators to their underlying structure enables an automated mixing analysis. This part of the framework can thus be seen on one hand as a unification of positional approaches to LFA by introducing structures that allow for the native treatment of arbitrary translational invariances and on the other hand as a general strategy for the automation of the frequency part back-end of LFA. In order to deal with compositions of operators as encountered in the analysis of iterative methods, the tools provided in~\cref{sec:crystal_reps_and_isos} allow for an automatic transformation of the underlying crystal structures into compatible representations and the corresponding transformations of the operators. Thus, the only task remaining for the user is to provide any, i.e., the simplest or most convenient, crystal representation of the operators. The following examples show how such a construction can be carried out and as such serve as a tutorial for the use of the algorithms in~\cref{sec:algorithms}. Annotated Jupyter Notebooks of the presented examples can be found in~\cite{aLFA_NK}.

	\subsection{Multicolored block smoother for the tight-binding Hamiltonian of graphene}\label{subsec:multicolor_graphene}
	
	In \cite{KahlKint2018} a multigrid method for the tight-binding Hamiltonian of the carbon allotrope graphene based on Kaczmarz smoothing is constructed and analyzed using conventional LFA. Due to the hexagonal structure of graphene, the lexicographic ordering of Kaczmarz and the mixing analysis of the coarse grid correction which involved a mixing of eight frequencies, the analysis turned out to be quite lengthy. In this subsection we now want to analyze a two-grid method for this problem where we replace Kaczmarz by an overlapping colored Gauss-Seidel method that allows for better parallelism in the application of the multigrid method. Thus, the goal of this example is two-fold, first we want to show that the tight-binding Hamiltonian can be easily expressed using the native crystal structure of graphene and second that even the analysis of an overlapping block smoother can be carried out with ease using aLFA due to the fact that only a representation of the involved operators is needed.

	The graphene structure can be described as a crystal $\mL^\se(\lA)$ where the underlying lattice is triangular, i.e., any three nearby lattice points form an equilateral triangle. We have
	\begin{align*}
		\mL^\se(\lA),\quad \la_1 = (\frac{3}{2}, \frac{\sqrt{3}}{2}) a,\quad \la_2 = (\frac{3}{2},-\frac{\sqrt{3}}{2}) a,
	\end{align*}
	with the structure element
	\begin{align*}
		\se = (\se_1,\se_2), \quad \se_1 = (\la,0)=\frac{1}{3}(\la_1+\la_2), \quad \se_2=(2\la,0)=\frac{2}{3}(\la_1+\la_2).
	\end{align*}
	The parameter $a=  1.42$\AA{} which denotes the distance of two neighboring atoms can be omitted as it does not occur in the tight-binding formulation. An illustration of the crystal has already been given in \cref{fig:crystal_example}.
	
	The (nearest neighbor) tight-binding Hamiltonian is a multiplication operator%

	\begin{align*}
		L: \fs(T_{\lA} ^\se) \rightarrow \fs(T_{\lA} ^\se), \quad (Lf)(x) := \sum_{y \in T_{\lA} } m_{L}^{(y)} f(x+y)
	\end{align*}
	with non-zero multipliers
	\begin{center}
		\begin{tikzpicture}
			
			\pgfmathsetmacro{\hshift}{2.2}
			\pgfmathsetmacro{\vshift}{1}
			
			\node (c) at (0,0) {$=$}; %
			\node[left=-.1cm of c, anchor=east] {$m_L^{(0)}$};
			\node[right=-.1cm of c, anchor=west] {$\mat[c]{0 & -1 \\ -1 & 0}$};
			
			\node (c) at (-\hshift,\vshift) {$=$}; %
			\node[left=-.1cm of c, anchor=east] {$m_L^{(-\la_2)}$};
			\node[right=-.1cm of c, anchor=west] {$\mat[c]{0 & -1 \\ 0 & 0}$};
			
			\node (c) at (\hshift,\vshift) {$=$}; %
			\node[left=-.1cm of c, anchor=east] {$m_L^{(\la_1)}$};
			\node[right=-.1cm of c, anchor=west] {$\mat[c]{0 & 0 \\ -1 & 0}$};
			
			\node (c) at (-\hshift,-\vshift) {$=$}; %
			\node[left=-.1cm of c, anchor=east] {$m_L^{(-\la_1)}$};
			\node[right=-.1cm of c, anchor=west] {$\mat[c]{0 & -1 \\ 0 & 0}$};
			
			\node (c) at (\hshift,-\vshift) {$=$}; %
			\node[left=-.1cm of c, anchor=east] {$m_L^{(\la_2)}$};
			\node[right=-.1cm of c, anchor=west] {$\mat[c]{0 & 0 \\ -1 & 0}$};
			
		\end{tikzpicture}\end{center}
	as illustrated in \cref{fig:tbh_stencil_and_4color_hexagon}, a).

	\begin{figure}
		\TikzFigsABresize[a)][b)]{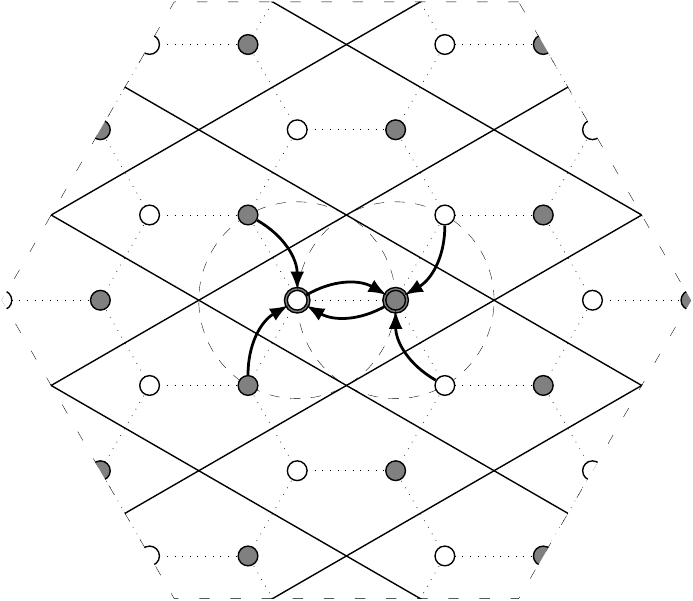}{graphene-4color-hexagon-bit}
		\caption{a) Schematic stencil of the tight binding Hamiltonian $L_{[t_0,t_1]}$ of graphene. b) Illustration of the domain decomposition into hexagons. Each unknown belongs to $3$ different colors.}\label{fig:tbh_stencil_and_4color_hexagon}
	\end{figure}

	\paragraph{Overlapping Hexagons}
	We now present an overlapping block smoother for the tight-binding Hamiltonian $L$.
	Consider the non-disjoint splitting (or coloring) of the crystal into hexagons as depicted in \cref{fig:tbh_stencil_and_4color_hexagon}, b). This splitting has a translational invariance of $\lC = 2\lA$. Rewriting the graphene crystal $T_\lA^\se$ with respect to this coarser lattice $\mL(\lC)$, we find $T_\lA^\se \cong T_\lC^\st$ with the structure element %
	\begin{align*}
		\st =(\st_1,\ldots,\st_8)= (\se, \se +\la_1, \se+\la_2, \se+\la_1+\la_2).
	\end{align*}
	The splitting is then given by the structure elements 
	\begin{align*}
		\st^{(1)}&= (\st_1^{(1)},\ldots,\st_6^{(1)})  = (\st_2,\st_3,\st_4,\st_5,\st_6,\st_7), \\
		\st^{(2)}&= \st^{(1)} + \la_1,\ %
		\st^{(3)}=  \st^{(1)} + \la_2,\text{ and } %
		\st^{(4)}=  \st^{(1)} + \la_1 + \la_2, %
	\end{align*}
	such that
	$T_\lC^{\st^{(1)}} \widehat{=} \CGcaptionAab[]$, %
	$T_\lC^{\st^{(2)}} \widehat{=} \CGcaptionAbb[]$,
	$T_\lC^{\st^{(3)}} \widehat{=} \CGcaptionAba[]$ and
	$T_\lC^{\st^{(4)}} \widehat{=} \CGcaptionAaa[]$.
	In the case of the nearest neighbor Hamiltonian any two hexagons of the same color do not interact. Thus, a relaxed sweep on the unknowns of one color is cheap and, in addition, yields a good degree of parallelism. 
	
	There are two options to describe the multiplication operators of the multicolored block smoother. On the one hand, we can construct them from scratch by defining the multiplication operator structure as well as the non-zero multipliers. On the other hand, we can derive the structure and non-zero multipliers of the operators corresponding to the colored blocking from the tight-binding Hamiltonian $L$ by exploiting the fact that a colored block corresponds to the operator $L$ restricted to this block. %
	In order to proceed with the latter approach, we first need to find the description $\hat{L}$ of the underlying operator, the tight-binding Hamiltonian $L$, with respect to the translational invariance of the splitting $\lC$. After that the structure element needs to be adjusted such that all unknowns and their couplings among each other are found within the central multiplier $m_{\hat{L}}^{(0)}$. Consider the structure element $\st$. As can be seen in \cref{fig:tbh_stencil_and_4color_hexagon}, b), the coupling among the unknowns $\st^{(\ell)}$ which we want to update simultaneously are found in the multipliers:
	
	\begin{center}
		\begin{tabular}{c|c|c|c}
			$\st^{(1)}$ &             $\st^{(2)}$ &  $\st^{(3)}$ & $\st^{(4)}$  \\
			\hline\rule{0pt}{3.1ex}
			$m_{\hat{L}}^{(0)}$ &  
			$m_{\hat{L}}^{(0)},\ m_{\hat{L}}^{(\pm \lc_1)}$ &  
			$m_{\hat{L}}^{(0)},\ m_{\hat{L}}^{(\pm\lc_2)}$ &   
			$m_{\hat{L}}^{(0)},\ m_{\hat{L}}^{(\pm\lc_1)},\ m_{\hat{L}}^{(\pm\lc_2)},\  m_{\hat{L}}^{(\pm(\lc_1-\lc_2))}$ \ %
		\end{tabular}
	\end{center}
	Thus, in order to obtain suitable descriptions for $\st^{(\ell)}$, $\ell\in\{2,3,4\}$, we need to consider shifted versions $\st + \tau^{(\ell)}$. For example
	\begin{equation}\label{eq:shiftedhex}
		\begin{bmatrix} \tau^{(1)}, & \tau^{(2)}, & \tau^{(3)}, & \tau^{(4)}\end{bmatrix} := \begin{bmatrix} 0, & \la_1, & \la_2, & \la_1 + \la_2 \end{bmatrix} .
	\end{equation}
	Finally, the error propagator corresponding to a single color relaxed block Gauss-Seidel update can be written as
	\begin{align*}
		G^{(\ell)} = (I - \omega (S^{(\ell)})^\dagger \hat{L}) : \fs(T_\lC^{\st + \tau^{(\ell)} }) \rightarrow \fs(T_\lC^{\st + \tau^{(\ell)} })
	\end{align*}
	with $(S^{(\ell)}f)(x) := (m_P m_{\hat{L}}^{(0)} m_P)f(x)$ where $m_P$ is the diagonal matrix 
	\begin{equation}\label{eq:mpii_hex}
		(m_P)_{ii} = \begin{cases} 1 & \text{ for }i\in\{2,3,4,5,6,7\} \\
			0 & \text{ else}.
		\end{cases}
	\end{equation}
	
	We summarize the algorithmical steps to obtain the error propagators.
	\medskip
	\exbox[Obtaining error propagators corresponding to splitting methods]{
		Obtain the description of the tight-binding Hamiltonian $L$ with respect to the translational invariance of the splitting via \cref{alg:mult_op_coarserning}:
		\begin{align*}
			(\hat{L})=\text{\textproc{LatticeCoarsening}}(L,\lC).
		\end{align*}
		Adjust the structure elements, such that the connections among the unknowns updated simultaneously are found within the central multiplier via~\cref{alg:struct_elem_change}:
		\begin{align*}
			(\hat{L}^{(\ell)})=\text{\textproc{ChangeStructureElement}}(\hat{L},\st+\tau^{(\ell)},\st+\tau^{(\ell)}),
		\end{align*}
		where $\tau^{(\ell)}$ is defined according to~\cref{eq:shiftedhex}.
		Using $m_P$ as defined in \cref{eq:mpii_hex} define the operators %
		\begin{align*}
			S^{(\ell)}: \fs(T_\lC^{(\st+\tau^{(\ell)})}) \rightarrow  \fs(T_\lC^{(\st+\tau^{(\ell)})}), \quad
			m_{S^{(\ell)}}^{(0)} := m_P m_{\hat{L}^{(\ell)}}^{(0)} m_P.%
		\end{align*}
	}
	\medskip
	The error propagator of a successive update of the four colors corresponds to the product of the error propagators
	$G^{(\ell)} = ( I - \omega (S^{(\ell)})^\dagger \hat{L}^{(\ell)})$.

	Now that all operators of the overlapping colored block Gauss-Seidel smoother are defined, we can compute its spectrum. %
	The computation of the eigenvalues of the error propagator is carried out with \cref{alg:compute_spectral_radii} via
	$$\textproc{ComputeSpectrum}(g,I,L,S^{(1)},\ldots,S^{(4)}).$$ The function $g$ denotes the composition of the error propagators
	\begin{align*}(I,L,S^{(1)},\ldots,S^{(4)}) \xmapsto{\enskip g\enskip } \prod_{\ell=1}^4 ( I - \omega (S^{(\ell)})^\dagger L)=:G.
	\end{align*} Note, that in this algorithm all operators are checked for compatibility and any incompatibility with respect to domains and codomains is dealt with using the transformations introduced in~\cref{sec:crystal_reps_and_isos}.
	For $\omega=\frac{1}{2}$ we obtain the plot in \cref{fig:hexagon_smooth_twogrid}, a). As the largest spectral radius is greater than one, this method cannot be used as a standalone solver.%
	\begin{figure}
		\phantom{}
		\TikzFigsABphantomresize[a)][b)]{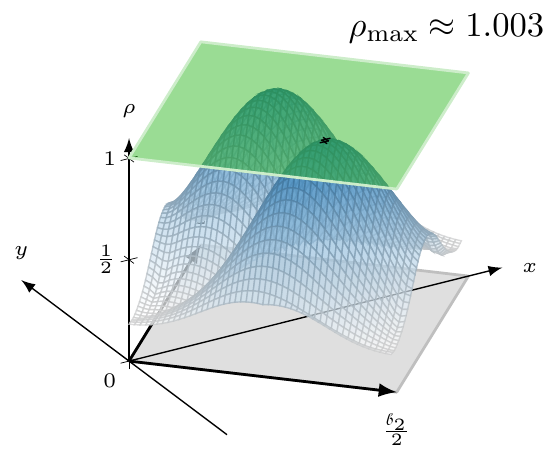}{LFA-tbh-4color-omega-05-twogrid}{LFA-tbh-4color-omega-05-v2}
		\caption{a) Plot of the spectral radii of the four-color hexagonal overlap block smoother $\prod_{\ell=1}^4 G_{\omega}^{(\ell)}$ with under-relaxation of $\omega = \frac12$ for the tight-binding Hamiltonian $L=L_{[0,-1]}$ plotted along $\PT(\frac{1}{2}\lB) = \PT((2\lA)^{-T})$. b) Plot of the spectral radii of the two-grid method using the overlap smoother.}\label{fig:hexagon_smooth_twogrid}
	\end{figure}

	\paragraph{Coarse grid correction} 
	In \cite{KahlKint2018} a Galerkin coarse grid correction is used with a corresponding coarse grid correction operator that is defined by%
	\begin{align*}
		E = (I- P(L_c)^\dagger R L)
	\end{align*}
	with $L_c = RLP$ and $P=R^T$. Just as in the description of the smoother, the shift invariance of the coarse grid is $\lC=2\lA$. Defining the coarse structure element by $2\se$, and with $\st$ denoting the structure element of the fine crystal according to the coarse lattice basis, the restriction operator can be described as
	\begin{align*}
		R : \fs(T^\st_{2\lA}) \rightarrow \fs(T^{2\se}_{2\lA}).
	\end{align*}
	The multipliers of the restriction operator according to \cite{KahlKint2018} then read

	\begin{center}
		\noindent\resizebox{.99\linewidth}{!}{
			\begin{tikzpicture}
				
				\pgfmathsetmacro{\hshift}{4}
				\pgfmathsetmacro{\vshift}{1.3}
				
				\node (c) at (0,0) {$=$}; %
				\node[left=-.1cm of c, anchor=east] {$m_R^{(0)}$};
				\node[right=-.1cm of c, anchor=west] {$\mat{
						0 & 1 & 0 & -\frac{1}{2} & 0 & -\frac{1}{2} & 0 & \frac{1}{4} \\
						\frac{1}{4} & 0 & -\frac{1}{2} & 0 & -\frac{1}{2} & 0 & 1 & 0}$};
				
				\node (c) at (0,2*\vshift) {$=$}; %
				\node[left=-.1cm of c, anchor=east] {$m_R^{(2\la_1-2\la_2)}$};
				\node[right=-.1cm of c, anchor=west] {$\mat{
						0 & 0 & 0 & 0 & 0 & \frac{1}{4} & 0 & 0 \\
						0 & 0 & 0 & 0 & \frac{1}{4} & 0 & 0 & 0}$};
				
				\node (c) at (-\hshift,\vshift) {$=$}; %
				\node[left=-.1cm of c, anchor=east] {$m_R^{(-2\la_2)}$};
				\node[right=-.1cm of c, anchor=west] {$\mat{
						0 & 0 & 0 & \frac{1}{4} & 0 & -\frac{1}{2} & 0 & -\frac{1}{2} \\
						0 & 0 & 0 & 0 & 0 & 0 & 0 & 0}$};
				
				\node (c) at (\hshift,\vshift) {$=$}; %
				\node[left=-.1cm of c, anchor=east] {$m_R^{(2\la_1)}$};
				\node[right=-.1cm of c, anchor=west] {$\mat{
						0 & 0 & 0 & 0 & 0 & 0 & 0 & 0 \\
						-\frac{1}{2} & 0 & \frac{1}{4} & 0 & -\frac{1}{2} & 0 & 0 & 0}$};
				
				\node (c) at (-1.75*\hshift,0) {$=$}; %
				\node[left=-.1cm of c, anchor=east] {$m_R^{(-2\la_1-2\la_2)}$};
				\node[right=-.1cm of c, anchor=west] {$\mat{
						0 & 0 & 0 & 0 & 0 & 0 & 0 & \frac{1}{4} \\
						0 & 0 & 0 & 0 & 0 & 0 & 0 & 0}$};
				
				\node (c) at (2.25*\hshift,0) {$=$}; %
				\node[left=-.1cm of c, anchor=east] {$m_R^{(2\la_1+2\la_2)}$};
				\node[right=-.1cm of c, anchor=west] {$\mat{
						0 & 0 & 0 & 0 & 0 & 0 & 0 & 0 \\
						\frac{1}{4} & 0 & 0 & 0 & 0 & 0 & 0 & 0}$};
				
				\node (c) at (-\hshift,-\vshift) {$=$}; %
				\node[left=-.1cm of c, anchor=east] {$m_R^{(-2\la_1)}$};
				\node[right=-.1cm of c, anchor=west] {$\mat{
						0 & 0 & 0 & -\frac{1}{2} & 0 & \frac{1}{4} & 0 & -\frac{1}{2} \\
						0 & 0 & 0 & 0 & 0 & 0 & 0 & 0}$};
				
				\node (c) at (\hshift,-\vshift) {$=$}; %
				\node[left=-.1cm of c, anchor=east] {$m_R^{(2\la_2)}$};
				\node[right=-.1cm of c, anchor=west] {$\mat{
						0 & 0 & 0 & 0 & 0 & 0 & 0 & 0 \\
						-\frac{1}{2} & 0 & -\frac{1}{2} & 0 & \frac{1}{4} & 0 & 0 & 0}$};
				
				\node (c) at (0,-2*\vshift) {$=$}; %
				\node[left=-.1cm of c, anchor=east] {$m_R^{(-2\la_1+2\la_2)}$};
				\node[right=-.1cm of c, anchor=west] {$\mat{
						0 & 0 & 0 & \frac{1}{4} & 0 & 0 & 0 & 0 \\
						0 & 0 & \frac{1}{4} & 0 & 0 &  0 & 0 & 0}$};
			\end{tikzpicture}
		}
	\end{center}
	This coarse grid correction is constructed in a way, such that it preserves the kernel of the tight-binding Hamiltonian $L$, where with $\lB=\mat{\lb_1 & \lb_2} = \lA^{-T}$, $$\operatorname{ker}(L)=\operatorname{span}\set{\mat{\cexp{K}{x} \\ 0}, \mat{0 \\ \cexp{K}{x}}}{K \in\{\frac{1}{3} \lb_1 + \frac{2}{3} \lb_2,\frac{2}{3} \lb_1 + \frac{1}{3} \lb_2\} }.$$ The frequencies corresponding to the kernel modes are known as the \emph{Dirac points}. %
	
	The two-grid analysis can now be carried out using \cref{alg:compute_spectral_radii} via
	$$\text{\textproc{ComputeSpectrum}}(f,I,L,S^{(1)},\ldots,S^{(4)}, R). $$  %
	The function $f$ denotes the composition of the two-grid error propagator
	\begin{align*}(I,L,S^{(1)},\ldots,S^{(4)},R) \xmapsto{\enskip f\enskip } GEG.%
	\end{align*}
	
	A plot of the spectral radii of the two-grid error propagator is given in b) of \cref{fig:hexagon_smooth_twogrid} which shows that the two-grid method converges with a convergence rate of $\rho_{\max} \approx 0.167$. Thus, this new method with overlapping colored block Gauss-Seidel smoothing not only yields opportunities for parallel computations, but also converges faster than the old approach which used Kaczmarz smoothing.

	In order to double-check the results of the developed theory, we show in~\cref{tbl:graphene_twogrid_results} that the asymptotic convergence rate of the two-grid method with random initial guess $x_0$ and right-hand-side $0$ coincides with the convergence rate obtained in the LFA with a relative accuracy of roughly $.002\%$. This comes as no surprise as the theory is exact for this problem with periodic boundary conditions and we have chosen the sampling of the frequency space in accordance with the problem size (cf.~\cref{rem:freq_sampling}).

	\begin{table}
		\begin{tabular}{cccc}
			\toprule
			iteration $i$ &        $||r_i||_2:=||b - A x_i||_2$ &             $\rho_i:=\frac{||r_i||_2}{||r_{i-1}||_2}$ &  $\frac{|\rho_i - \rho_{\text{analytic}}|}{\rho_{\text{analytic}}}$ \\
			\midrule
			398 &  1.161369e-312 &  0.16685534 &      0.00220\% \\
			399 &  1.937806e-313 &  0.16685539 &      0.00217\% \\
			400 &  3.233335e-314 &  0.16685545 &      0.00213\% \\
			\bottomrule
		\end{tabular}
		\caption{Convergence history of the two-grid method applied to the tight-binding Hamiltonian of graphene with $(41\times 41)\cdot 2^3$ unknowns/atoms and periodic boundary conditions. The reported asymptotic convergence rate $\rho_i$ coincides with high precision to the convergence estimate $\rho_{\text{analytic}} = \sup_{k \in T^*_{2\lA,41\cdot 2\lA}} \set{| \lambda |}{\lambda \text{ eigenvalue of } G_kE_kG_k} = 0.16685901$ obtained in aLFA.}\label{tbl:graphene_twogrid_results}
	\end{table}
	
	\subsection{Two-level analysis for the curl-curl equation}\label{subsec:curlcurl}
	
	In~\cite{BoonLentVand2008} a complete two-level analysis is carried out for the $2$-dimensional curl-curl formulation of Maxwell's equations using conventional LFA. In this subsection we want to reproduce the results of the two-grid method discussed in this paper making use of the native crystal structure of the staggered discretization of the curl-curl equations. The method consists of a so-called half-hybrid smoother, introduced in \cite{Hipt1998}, and a Galerkin coarse grid correction. %
	
	The degrees of freedom of the discrete curl-curl equation
	\begin{align}\label{eq:discrete_curl}
		(\frac{1}{h^2}K_{cc} + \sigma M)x= b
	\end{align}
	are associated with the edges of a quadrilateral lattice $\mL(\frac{1}{h} \lA)$ with  $$\lA = \mat{\la_1 &  \la_2}= \mat{1 & 0 \\ 0 & 1}.$$
	By multiplying \cref{eq:discrete_curl} with $h^2$, the grid size can be incorporated into the material parameter $\sigma_h := h^2\sigma$. Then the discrete operator 
	$K=K_{cc} + \sigma_h M$ can be expressed as a multiplication operator by
	\begin{align*}
		K: \fs(T_\lA^\see) \rightarrow \fs(T_\lA^\see) \quad \text{with} \quad 
		\see = (\see_h,\see_v) = (\frac{1}{2}\la_1, \frac{1}{2}\la_2),
	\end{align*}
	where the elements $\mL(\lA) + \see_h$ and $\mL(\lA) + \see_v$ correspond to the (midpoints of the) horizontal and the vertical edge, respectively (cf.~\cref{fig:curlcurl_crystal}).
	\begin{figure}
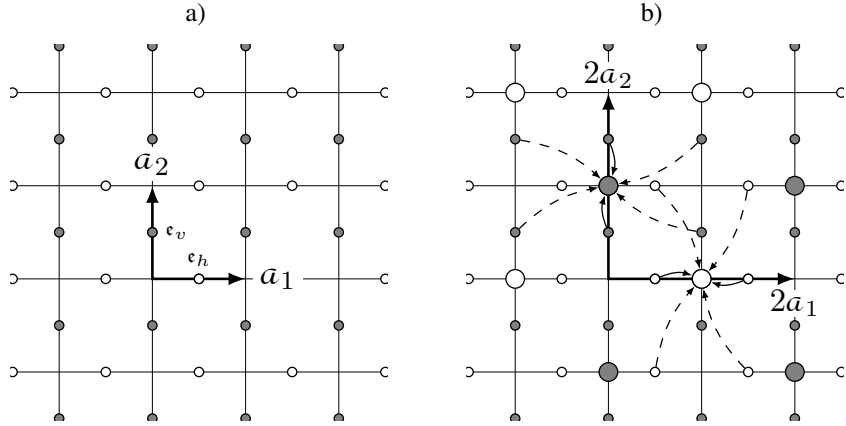

		\centering
		\TikzFigsABresize{curlcurl-crystal}{curlcurl-restriction}
		\caption{a) Crystal structure of the curl-curl discretization. b) Illustration of the coarse crystal and the restriction operator; dashed lines correspond to restriction weight $\frac14$, solid lines to $\frac12$.}
		\label{fig:curlcurl_crystal}
	\end{figure}
	According to \cite{BoonLentVand2008} the multipliers are given by:
	\begin{center}
		\resizebox{.9\textwidth}{!}{
			$ 
			\begin{array}{rccrccrcc}
				m_K^{(-\la_1+\la_2)} &= &\mat{0 & 0\\-1 & 0}&
				m_K^{(\la_2)} &= &\mat{-1+\frac16 \sigma_h & 0 \\ 1 & 0} & \\[1em]
				m_K^{(-\la_1)} &= &\mat{0 & 0 \\ 1 & -1+\frac16 \sigma_h} &
				m_K^{(0)} &= &\mat{2+\frac23\sigma_h & -1 \\ -1 & 2+\frac23\sigma_h} &
				m_K^{(\la_1)} &= &\mat{0 & 1 \\ 0 & -1+\frac16 \sigma_h} \\[1em]
				& & &
				m_K^{(-\la_2)} &= &\mat{ -1+\frac16 \sigma_h & 1 \\ 0 & 0} &
				m_K^{(\la_1-\la_2)} &= &\mat{0 & -1 \\ 0 & 0}
			\end{array}
			$}%
	\end{center}

	\paragraph{Hybrid smoother}

	The (half) hybrid smoother uses an auxiliary space correction. That is, it can be seen as a two-grid method itself where smoothing replaces the exact inversion of the auxiliary space operator. The construction of this operator is as follows.

	As already stated, we want to reproduce the results in~\cite{BoonLentVand2008}. Due to the fact that lexicographic Gauss-Seidel smoothers are used, we introduce the exact same ordering of the lattice points used in this paper, that is, the coordinates $y\in \mL(\lA)$ are ordered from bottom to top and left to right which leads us to the following definition for $x,y\in\mL(\lA)$:
	$$x=(x_1,x_2) < y=(y_1,y_2) :\Leftrightarrow x_2 < y_2 \text{ or } \left(x_2 = y_2 \text{ and } x_1 < y_1\right) .$$
	
	The error propagator of the smoother on the original crystal $T_\lA^\see$ is a node-based lexicographic Gauss-Seidel iteration. When updating a horizontal edge $x + \see_h$, $x\in\mL(\lA)$ it is assumed that all edges $y+\see_h$ and $y+\hat{\see}_v$, $\hat{\see}_v:=\see_v+\la_1 - \la_2$, with $y<x$ are already updated. When updating the edge $y+\hat{\see}_v$, it is additionally assumed that $x+\see_h$ is already updated (cf.~\cite[Figure 6.1]{BoonLentVand2008}). In order to obtain the error propagator which exactly represents this ordering, it is convenient to rewrite the operator $K$ with respect to this representation of the structure element, that is
		\begin{align*}
			K \cong \hat{K}: \fs(T_\lA^{\hat{\see}}) \rightarrow \fs(T_\lA^{\hat{\see}}) \quad \text{with} \quad 
			\hat{\see} = (\see_h,\hat{\see}_v) \cong \see,
	\end{align*}
	which can be obtained with \cref{alg:struct_elem_change} via 
	\begin{align*}
			\hat{K}=\text{\textproc{ChangeStructureElement}}(K,\hat{\see},\hat{\see}).%
		\end{align*}
	The corresponding error propagator is then given by
	\begin{align*}
		G_E = (I - S_E^\dagger K ): \fs(T_\lA^\see) \rightarrow \fs(T_\lA^\see)
	\end{align*}
	with the multipliers
	\begin{align*}
		m_{S_E}^{(y)} = 
		\begin{cases}
			m_{\hat{K}}^{(y)} & \text{if } y < 0, \\
			\texttt{tril}(m_{\hat{K}}^{(y)}) & \text{if } y=0, \\
			0 & \text {else.}
		\end{cases}
	\end{align*}
	In here, only the lower triangular part $\texttt{tril}(m_{\hat{K}}^{(0)})$ of the central multiplier is used due to the fact that it represents a Gauss-Seidel sweep where a horizontal edge is updated before a vertical edge.
	
	The crystal, where the auxiliary space system is formulated, is given by $T_\lA^{(0)}$, i.e., the nodal points between the edges. The transfer operator to this crystal is the discrete gradient operator defined by
	$R_N: \fs(T_{\lA}^\see) \rightarrow \fs(T_{\lA}^{(0)})$ 
	with non-zero multipliers
	\begin{align*}
		\begin{array}{rccrcc}
			m_{R_N}^{(-\la_1)} &=& \mat{1 & 0} & m_{R_N}^{(0)} &=& \mat{-1 & -1} \\[1em]
			& & & m_{R_N}^{(-\la_2)} &= &\mat{0 & 1}.
		\end{array}
	\end{align*}
	Using a Galerkin construction, i.e., $P_N = R_N^T$, the coarse grid operator $K_N = R_N K P_N$ can be obtained using the computation rules in~\cref{lem:calc_rules_mult_op}. Inversion of this auxiliary space operator is then approximated by Gauss-Seidel. Thus, the auxiliary space correction with a single smoothing step on the auxiliary space is given by
	\begin{align*}
		G_N = (I - P_N S_N^\dagger R_N K) : \fs( T_\lA^\see) \rightarrow \fs(T_\lA^\see),\end{align*} %
	where $S_N: \fs(T_\lA^{(0)}) \rightarrow \fs(T_\lA^{(0)})$ consists of the scalar multipliers
	\begin{align*}
		m_{S_N}^{(y)} = \begin{cases} m_{K_N}^{(y)} & \text{if } y \leq 0, \\
			0 & \text {else.}\end{cases}
	\end{align*}
	With this definition of the auxiliary space correction, the half-hybrid smoother consists of the following steps.
	
	\begin{enumerate}
		\item Smooth on $Kx=b$: $x \leftarrow (I+S_E^\dagger K)x + S_E^\dagger b$,
		\item Restrict the residual: $r_N \leftarrow R_N(b-Kx)$,
		\item Smooth on $K_N x_N= r_N$ with zero initial guess: $x_N \leftarrow S_N^\dagger r_N$,
		\item Prolongate to the primal crystal $x \leftarrow x + P_N x_N$.
	\end{enumerate}

	The smoother is analyzed analogously to \cref{subsec:multicolor_graphene} with \cref{alg:compute_spectral_radii} by computing the spectral radii via
	$$\text{\textproc{ComputeSpectrum}}(g,I,K,S_E,R_N,S_N),$$
	where $g$ denotes the composition of the error propagators
	$$(I,K,S_E,R_N,S_N) \xmapsto{\enskip g\enskip } G_N G_E =:G.
	$$
	In \cref{fig:LFA_curlcurl} a) a contour plot of $\rho(G_k)$ with $\sigma_h = 0.01$, with respect to $k\in \lA^{-T}[ -\frac{1}{4},\frac{3}{4})$ is given. It corresponds to the result given in \cite[Figure 6.2, right]{BoonLentVand2008}. %
	
	\medskip 
	\paragraph{Coarse grid correction} In \cite{BoonLentVand2008} a Galerkin coarse grid correction is used corresponding to the error propagator
	\begin{align*}
		E = (I- P(K_c)^\dagger R K)
	\end{align*}
	with $K_c = RKP$ and $P=R^T$. In here, the coarse crystal is $\mL^{2\see}(2\lA)$ and the original crystal $\mL^{\see}(\lA)$ with respect to the lattice $2\lA$ is given by $\mL^{\sef}(2\lA)$ with
	\begin{align*}
		\sef = (\see , \see + \la_1, \see + \la_2, \see + \la_1 + \la_2).
	\end{align*}
	Then, according to \cite{BoonLentVand2008} the restriction operator is given as
	$R : \fs(T^\sef_{2\lA}) \rightarrow \fs(T^{2\se}_{2\lA})$ 
	with the multipliers:
	\begin{center}
		\resizebox{.7\textwidth}{!}{
			$
			\begin{array}{rccrcc}
				m_R^{(-2\la_1)} &=& 
				\mat{0 & 0 & 0 & 0 & 0 & 0 & 0 & 0 \\
					0 & 0 & 0 & \frac14 & 0 & 0 & 0 & \frac14 }&
				m_R^{(0)} &= &\mat{
					\frac{1}{2} & 0 & \frac{1}{2} & 0 & \frac{1}{4} & 0 & \frac{1}{4} & 0\\
					0 & \frac{1}{2} & 0 & \frac{1}{4} & 0 & \frac{1}{2} & 0 & \frac{1}{4} } \\[1em]
				& & & m_R^{(-2\la_2)} &=& 
				\mat{0 & 0 & 0 & 0 & \frac14 & 0 & \frac14 & 0 \\
					0 & 0 & 0 & 0 & 0 & 0 & 0 & 0 } 
			\end{array}
			$%
		}
	\end{center}

	Here, each coarse horizontal and vertical edge is connected to its six nearest edges of the same type (cf.~\cref{fig:curlcurl_crystal} b)). The prolongation and coarse grid operator $P$ and $K_c$ can be obtained via \cref{lem:calc_rules_mult_op}.
	
	The spectral radii of the two-grid method can then be obtained via 
	\begin{align*}
		\text{\textproc{ComputeSpectrum}}(f,I,K,S_E,R_N,S_N,R) %
	\end{align*}
	where $f$ denotes the composition of the two-grid error propagator
	\begin{align*}
		(I,K,S_E,R_N,S_N,R) &\xmapsto{\enskip f\enskip } G E G =: M.
	\end{align*}
	\Cref{fig:LFA_curlcurl} b) shows the spectral radii, $\sup_k \rho(M_k)$, of the two-grid error propagator as a function of $\sigma_h$. We used exactly the same orderings in pre- and post-smoothing, as described in~\cite{BoonLentVand2008}. Furthermore, we double-checked the results with convergence tests similar to~\cref{tbl:graphene_twogrid_results}. However, the obtained results show major differences to the one in~\cite[Figure 8.1, left]{BoonLentVand2008}, which leads us to believe that our assumptions on the lexicographic orderings employed in pre- or post-smoothing in~\cite{BoonLentVand2008} are wrong, as these have a large impact on the convergence rates.

	\begin{figure}
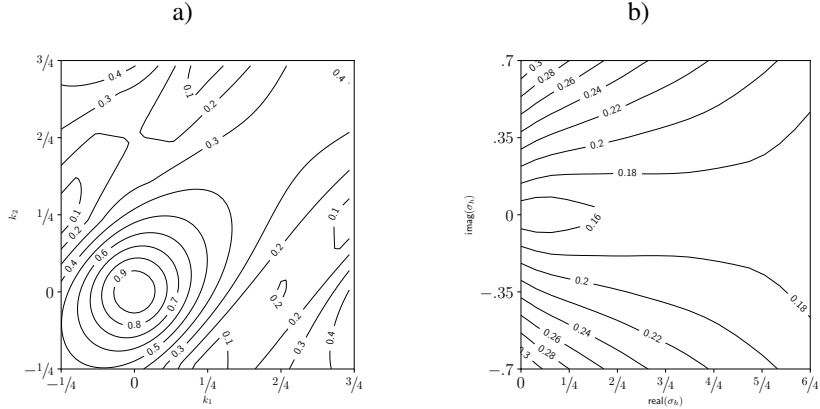

		\TikzFigsABresize{curlcurl-fixed-sigma_revision}{curlcurl-sigma-varying_revision_fixedlabel}
		\caption{a) Contour plot of $\rho(G_k)$ for the hybrid smoother with respect to $k\in\lA^{-T}[ -\frac{1}{4},\frac{3}{4})$, $\lA^{-T}=\lA$, $\sigma_h = 0.01$. b) Contour plot of the convergence estimates of the two-grid method $\sup_k \rho(M_k)$ with respect to $\sigma_h$. The plots correspond to the results \cite[Figure 6.2, right]{BoonLentVand2008} and \cite[Figure 8.1, left]{BoonLentVand2008}, respectively.}
		\label{fig:LFA_curlcurl}
	\end{figure}

	\section{Conclusion}
	In this paper we present an LFA framework that is applicable to arbitrary crystal structures, which are encountered in many applications, e.g., systems of partial differential equations, block smoothers or tight-binding formulations. This was achieved by introducing a rigorous notion of crystal structures and translationally invariant operators that manipulate value distributions on crystal structures. Based on these structures we introduced a complete framework to modify and transform these operators with respect to different crystal representations by using normal forms of integer linear algebra. This allowed us to automate the LFA in both position space, i.e., in terms of multiplication operators/stencils, and frequency space, i.e., in terms of canonical basis functions and samplings. In two examples we showed that the approach can be used for complicated operators, i.e., hexagonal grids, overlapping block smoothers and hybrid smoothers, without requiring insight into the frequency space back-end of LFA. An explicit mixing calculation is no longer needed, and the user only has to provide any representation of the individual operators. The transformation of the individual operators to a compatible representation and the subsequent frequency analysis is then carried out automatically. Even though we have limited ourselves in the examples in this paper to $2$-dimensional problems our automated LFA can be applied to operators in higher dimension, where the difference is a larger set of primitive vectors defining the underlying lattice.
	
	The automation presented in this paper does have some limitation. Each individual operator in the analysis is only allowed to change each value of the value distribution at most once. This limitation solely restricts the class of smoothers that can be analyzed with this approach. Any sequential, i.e., lexicographic, smoother with overlapping update regions changes values in the overlap multiple times in one application. This cannot be easily translated to a corresponding local multiplication operator, but it can be dealt with in frequency space (cf.~\cite{MacLOost2011,Molenaar1991,Sivaloganathan:1991:ULM}). This particular treatment of sequential overlap is momentarily not covered in our framework. %
	Note, that the mere presence of overlap is not the problem here. By introducing a coloring, such that the complete sweep can be split into a sequence of updates where each one of them only changes values at most once, automated LFA can be applied (cf.~\cref{subsec:multicolor_graphene}). Due to the fact that a coloring in overlapping approaches also favors parallelism over their sequential counterparts, we feel that this limitation is relatively minor when targeting actual applications. %
	An open-source implementation of the automated LFA framework~\cite{aLFA_NK} is freely available on GitLab\footnote{\href{https://gitlab.com/NilsKintscher/alfa}{\texttt{gitlab.com/NilsKintscher/alfa}}}. Jupyter Notebooks for all examples from~\cref{sec:example,subsec:multicolor_graphene,subsec:curlcurl} are included in this software package as well.

	\appendix
	\section{Rules of computation}\label{sec:rules_of_comp} %
	Calculus of multiplication operators plays a key-role in local Fourier analysis. In this section we list all elementary operations, such as addition and multiplication. Proofs for these rules can be obtained by straightforward calculation.
	\begin{lemma}
		\label{lem:calc_rules_mult_op}
		Let two multiplication operators be given by
		\begin{align*}
			\begin{array}{rcl}
				L:\fs(T_{\lA  ,\lZ }^{\se}) \rightarrow \fs(T_{\lA  ,\lZ }^{\st}), &
				(Lf)(x) = \displaystyle\sum_{y\in T_{\lA  ,\lZ }}  m_L^{(y)} f(x+y),& m_L^{(y)} \in\mC^{|{\st}|\times |{\se}|},\\
				G:\fs(T_{\lA  ,\lZ }^{\seu})\rightarrow \fs(T_{\lA  ,\lZ }^{\sev}),& %
				(Gf)(x) = \displaystyle\sum_{y\in T_{\lA  ,\lZ }}  m_G^{(y)} f(x+y), & m_G^{(y)} \in\mC^{|{\sev}|\times |{\seu}|}.
			\end{array}
		\end{align*}
		Then the following operators are multiplication operators as well:
		\begin{enumerate}%
			\item[(i)] If $\se = \seu$ and $\st=\sev$, then 
			$
			L+G:\fs(T_{\mc{A},\mc{Z}}^\se) \rightarrow \fs(T_{\mc{A},\mc{Z}}^\st) \text{ with } m_{L+G}^{(y)} =  m_L^{(y)}+m_G^{(y)}.$ \smallskip
			\item[(ii)] If $\sev = \se$, then 
			$LG:\fs(T_{\mc{A},\mc{Z}}^{\seu}) \rightarrow \fs(T_{\mc{A},\mc{Z}}^{{\st}}) \text{ with } m_{LG}^{(z)} = \sum\limits_{y+w=z} m_L^{(y)} \cdot m_G^{(w)}.$%
			\item[(iii)] The adjoint is given by
			$L^*:\fs(T_{\mc{A},\mc{Z}}^\st) \rightarrow \fs(T_{\mc{A},\mc{Z}}^\se) \text{ with }m_{L^*}^{(y)} = (m_{L}^{(-y)})^*.$
		\end{enumerate}
	\end{lemma}

	\begin{theorem}
		\label{thm:Eigenvalues_Crystal_Mult_Op}
		Let two multiplication operators be given by
		\begin{align*}
			\begin{array}{rcl}
				L:\fs(T_{\lA  ,\lZ }^{\se}) \rightarrow \fs(T_{\lA  ,\lZ }^{\st}), &
				(Lf)(x) = \displaystyle\sum_{y\in T_{\lA  ,\lZ }}  m_L^{(y)} f(x+y),& m_L^{(y)} \in\mC^{|{\st}|\times |{\se}|},\\
				G:\fs(T_{\lA  ,\lZ }^{\seu})\rightarrow \fs(T_{\lA  ,\lZ }^{\sev}),& %
				(Gf)(x) = \displaystyle\sum_{y\in T_{\lA  ,\lZ }}  m_G^{(y)} f(x+y), & m_G^{(y)} \in\mC^{|{\sev}|\times |{\seu}|}
			\end{array}
		\end{align*}
		with corresponding symbols $L_k$ and $G_k$.
		Then we have the following statements. \smallskip
		\begin{enumerate}%
			\item[(i)] Assuming that ${\se} = {\seu}$ and ${\st} = {\sev}$, the symbols of $L+G$ are given by $L_k+ G_k$. \smallskip%
			\item[(ii)] Assuming that ${\sev} = {\se}$, the symbols of $L\cdot G$ are given by $L_k\cdot G_k$. \smallskip%
			\item[(iii)] The symbols of $L^*$ are given by $L_k^*$. \smallskip
			\item[(iv)] The symbols $(L^\dagger)_k$ are given by $(L_k)^\dagger$.%
		\end{enumerate}
	\end{theorem}
	\section{Algorithms}\label{sec:algorithms} %
	\begin{algorithm2e}[H]{\footnotesize
			\caption{Spectrum of a composition of multiplication operators.}\label{alg:compute_spectral_radii}
			\SetAlgorithmStyle
			\Input{$L^{(j)}:T_{\lA^{(j)}}^{\se^{{(j)}}} \rightarrow T_{\lA^{(j)}}^{\st^{{(j)}}}, m_{L^{(j)}}^{(x)} \in \mC^{|{\st^{{(j)}}}|\times |{\se^{{(j)}}}|}, x\in\mL(\lA^{(j)})$ and composition function $f$.
			}
			\Output{Spectra of $f(L^{(1)}_k,\ldots,L^{(K)}_k)$, $k\in\PT(\lA^{-T})$}\smallskip
			\Fn{$\func{ComputeSpectrum}{f,L^{(1)},\ldots,L^{(K)}}$}{ 
				$(\hat{L}^{(1)},\ldots,\hat{L}^{(K)})=\func{MakeOperatorsCompatible}{L^{(1)},\ldots,L^{(K)}}$ \Comment*{See \cref{alg:Make_Operators_Consistent}}%
				Sample the dual lattice to obtain $k\in \PT(\lA^{-T})$ \Comment*{See \cref{rem:freq_sampling}}%
				Compute the spectra of the composition operator of symbols %
				\smallskip
				\nonl \quad\quad\quad\quad\quad\quad\quad\quad\quad\quad\quad\quad  
				$f(L^{(1)}_k,\ldots,L^{(K)}_k)$ \Comment*{See \cref{thm:Eigenvalues_Crystal_Mult_Op}}
			}
		}\end{algorithm2e}
	\clearpage
	\begin{algorithm2e}[H]{\footnotesize
			\caption{Normalize a multiplication operator.}\label{alg:mult_op_normalize}
			\SetAlgorithmStyle
			\Input{$L:T_{\lA}^{\se} \rightarrow T_{\lA}^{\st}, m_L^{(x)} \in \mC^{|{\st}|\times |{\se}|}, x\in\mL(\lA)$.}
			\Output{$G \cong L$ in \emph{normal form} with
				$G:T_{\lA}^{\seu} \rightarrow T_{\lA}^{\sev},  m_G^{(x)} \in \mC^{|{\sev}|\times |{\seu}|}, x\in\mL(\lA).$}\smallskip
			\Fn(\hfill $\triangleright$ See \cref{def:crystal_op_and_normalform}){$G = \func{Normalize}{L}$}{ 
				$\seu_j = \se_j - \lA \lfloor \lA^{-1}\se_j \rfloor $ for all $j=1,\ldots,|\se|$ \Comment*{Shift into $\PT(\lA) = \lA[0,1)^n$}%
				$\sev_j = \st_j - \lA \lfloor \lA^{-1}\st_j \rfloor $ for all $j=1,\ldots,|\st|$ \Comment*{Shift into $\PT(\lA) = \lA[0,1)^n$}%
				Sort $\seu$ and $\sev$ lexicographically\;%
				$G = \func{ChangeStructureElement}{L,\seu,\sev}$ \Comment*{See \cref{alg:struct_elem_change}}
			}
		}\end{algorithm2e}
	\begin{algorithm2e}[H]{\footnotesize
			\caption{Find all elements in the quotient space of two lattices.}\label{alg:elements_in_quotient_space}
			\SetAlgorithmStyle
			\Input{Two $n$-dimensional lattices with $\mL(\lC), \mL(\lA)$ with $\mL(\lC) \subset \mL(\lA)$.}
			\Output{Structure element $\se \cong T_{\lA,\lC} = \bigslant{\mL(\lA)}{\mL(\lC)}$}\smallskip
			\Fn(\hfill $\triangleright$ See \cref{cor:quotient_latticepointlist}){$\se= \func{ElementsInQuotientSpace}{\lA,\lC}$}{ %
				$H = $ Hermite normal form of $\lA^{-1}\lC$ \Comment*{See \cref{thm:integer_normalforms}}%
				$m = \prod_{i=1}^n h_{i,i}$ \Comment*{Size of $\se$}%
				$\se_i = 0$ for all $i=0,\ldots,m$ \Comment*{Initialize $\se$}%
				\For{$i=1,2,\ldots,m$}{
					$k=i-1$\;
					\For{$j=1,2,\ldots,n$}{
						$t = \operatorname{mod}(k,h_{j,j})$ \Comment*{Shift in direction $\la_j$}%
						$k=\frac{k-t}{h_{j,j}}$\;
						$\se_i = \se_i + t \la_j$\;
					}
				}
			}
		}\end{algorithm2e}
	\begin{algorithm2e}[H]{\footnotesize
			\caption{Rewrite a multiplication operator w.r.t.\ a coarser lattice.}\label{alg:mult_op_coarserning}
			\SetAlgorithmStyle
			\Input{$L:T_{\lA}^{\se} \rightarrow T_{\lA}^{\st}, m_L^{(x)} \in \mC^{|{\st}|\times |{\se}|}, x\in\mL(\lA),$ and a sublattice\ $\mL(\lC) \supset \mL(\lA)$}
			\Output{$G\cong L$ with 
				$G:T_{\lC}^{\seu} \rightarrow T_{\lC}^{\sev}, m_G^{(x)} \in \mC^{|{\sev}|\times |{\seu}|}, x\in\mL(\lC).$}\smallskip
			\Fn(\hfill $\triangleright$ See \cref{thm:operator_sublattice_coarsening}){$G = \func{LatticeCoarsening}{L,\lC}$}{%
				$\see= \func{ElementsInQuotientSpace}{\lA,\lC}$ \Comment*{See \cref{alg:elements_in_quotient_space}}%
				$\seu = (\see_1 + \se ,\ldots, \see_{|\see|} + \se), \sev = (\see_1 + \st ,\ldots, \see_{|\see|} + \st)$     \Comment*{Define structure elements}
				$(m_{{G}}^{({y})})=0 \in \mC^{|\sev|\times |\seu|}$ for all ${y}\in\mL(\lC)$ \Comment*{Initialize new multipliers}
				\For{$m_L^{(y)} \neq 0$}{
					\For{$(i,j) \in \{1,\ldots,|\see|\}^2$}{ %
						$(m_{G}^{(\lC\lfloor \lC^{-1}(y+{\see}_i-{\see}_j) \rfloor)})_{i,j} = m_L^{(y)}$ \Comment*{Define multipliers block-wise}
					}
				}
			}
		}\end{algorithm2e}
	\begin{algorithm2e}[H]{\footnotesize
			\caption{Rewriting multiplication operators w.r.t.\ a single lattice.}\label{alg:Make_Operators_Consistent}
			\SetAlgorithmStyle
			\Input{$L^{(j)}:T_{\lA^{(j)}}^{\se^{{(j)}}} \rightarrow T_{\lA^{(j)}}^{\st^{{(j)}}}, m_{L^{(j)}}^{(x)} \in \mC^{|{\st^{{(j)}}}|\times |{\se^{{(j)}}}|}, x\in\mL(\lA^{(j)})$
			}%
			\Output{$\hat{L}^{(j)}\cong L^{(j)}$ in \emph{normal form}, $\hat{L}^{(j)}:T_{\lA}^{\seu^{{(j)}}} \rightarrow T_{\lA}^{\sev^{{(j)}}},  m_{\hat{L}^{(j)}}^{(x)} \in \mC^{|{\sev^{{(j)}}}|\times |{\seu^{{(j)}}}|}, x\in\mL(\lA).%
				$%
			}\smallskip
			
			\Fn{$(\hat{L}^{(1)},\ldots,\hat{L}^{(K)}) = \func{MakeOperatorsCompatible}{%
					{L}^{(1)},\ldots,{L}^{(K)}}$}{
				$\lA = \lA^{(1)}$\;
				\For{$j=2,\ldots,K$}{
					$\lA = \func{LeastCommonMultiple}{\lA, \lA^{(j)}}$ \Comment*{See \cref{alg:lattice_lcm}}
				}
				\For{$j=1,\ldots,K$}{
					$\hat{L}^{(j)} = \func{LatticeCoarsening}{L^{(j)}, \lA}$ \Comment*{See \cref{alg:mult_op_coarserning}}
					$\hat{L}^{(j)} = \func{Normalize}{\hat{L}^{(j)}}$ \Comment*{See \cref{alg:mult_op_normalize}}%
				}%
			}
		}
	\end{algorithm2e}
	\begin{algorithm2e}{\footnotesize
			\caption{Changing the structure elements.}%
			\label{alg:struct_elem_change}
			\SetAlgorithmStyle
			\Input{Structure elements $\seu \cong \se$, $\sev \cong \st$ w.r.t.\ $\mL(\lA)$ and
				$L:T_{\lA}^{\se} \rightarrow T_{\lA}^{\st},\ m_L^{(x)} \in \mC^{|{\st}|\times |{\se}|}, x\in\mL(\lA)$.
			}
			\Output{$\hat{L} \cong L$ with $\hat{L}:T_{\lA}^{\seu} \rightarrow T_{\lA}^{\sev}, m_{\hat{L}}^{(x)} \in \mC^{|{\sev}|\times |{\seu}|}, x\in\mL(\lA)$.
			}\smallskip

			\Fn(\hfill $\triangleright$ See \cref{thm:shifted_structure_element,thm:permuted_structure_element}){$\hat{L} = \func{ChangeStructureElement}{L,{\seu},{\sev}}$}{
				$m_\pi =0 \in \{0,1\}^{|\se|\times|\se| }$,  $m_\sigma = 0 \in \{0,1\}^{|\st|\times|\st| }$ \Comment*{Initialize permutation matrices}%
				\For(\hfill $\triangleright$ Compute changes in $\sed$){$(i,j) \in \{1,\ldots,|\se|\}^2$}{%
					\If{$\lA^{-1}( \se_i - \seu_j)$ is integral}{ 
						$\see_i=\se_i - \seu_j$ \Comment*{Save shift}%
						$(m_\pi)_{j,i} = 1$ \Comment*{Save permutation}%
					}
				}
				\For(\hfill $\triangleright$ Compute changes in $\sec$){$(i,j) \in \{1,\ldots,|\st|\}^2$}{%
					\If{$\lA^{-1}( \st_i - \sev_j)$ is integral}{ 
						$\sef_i=\st_i - {\sev}_j$ \Comment*{Save shift}%
						$(m_\sigma)_{j,i} = 1$ \Comment*{Save permutation}%
					}
				}
				$(m_{\hat{L}}^{({y})})=0 \in \mC^{|\st|\times |\se|}$ for all ${y}\in\mL(\lA)$ \Comment*{Initialize new multipliers}
				\For{$m_L^{(y)} \neq 0$}{
					\For{$(i,j) \in \{1,\ldots,|\st|\}\times\{1,\ldots,|\se|\}$}{ 
						$(m_{\hat{L}}^{(y - \see_j + \sef_i )})_{i,j} = (m_{L}^{(y )})_{i,j} $  \Comment*{Incorporate shifts $\see$ and $\sef$, see \cref{thm:shifted_structure_element}}%
					}
				}
				$m_{\hat{L}}^{(y)} =  m_\sigma \cdot  m_{\hat{L}}^{(y)}\cdot m_\pi^{-1}$ for all ${y}\in\mL(\lA)$  \Comment*{Incorporate permutations, see \cref{thm:permuted_structure_element}}%
			}
		}
	\end{algorithm2e}
	\begin{algorithm2e}{\footnotesize
			\caption{Find a least common multiple of two lattices.}\label{alg:lattice_lcm}
			\SetAlgorithmStyle
			\Input{Lattice basis $\lB,\lA \in \mR^{n\times n}$}
			\Output{Lattice basis $\lC$, s.t. $\mL(\lC) \subset \mL(\lA)$ and $\mL(\lC) \subset \mL(\lB)$}\smallskip
			\Fn(\hfill $\triangleright$ See \cref{thm:lcm_mat}){$\lC = \func{LeastCommonMultiple}{\lA,\lB}$}{%
				Find an integer $r$, s.t. $M=r\lA^{-1}\lB$ is integral\;
				Compute Smith normal form $S=V^{-1}MT^{-1}$ of $M$ \Comment*{See \cref{thm:integer_normalforms}}%
				$(N_{\lB})_{i,i} = r \cdot \operatorname{gcd}(r,s_{i})^{-1}$, 
				$\lC = \lB T^{-1} N_{\lB}$ \Comment*{Define the lattice basis}%
			}
		}\end{algorithm2e}

	\bibliographystyle{spmpsci}
	
	\bibliography{Manuscript} %
	 
\end{document}